\documentclass{article}
\usepackage[english]{babel}
\usepackage{amsmath,amssymb,latexsym,bbm,tikz-cd,amsthm}
\usepackage{indentfirst}
\newcommand{\assign}{:=}

\newcommand{\dotminus}{\mathaccent95{-}}
\newcommand{\nobracket}{}
\newcommand{\of}{:}
\newcommand{\tmem}[1]{{\em #1\/}}
\newcommand{\tmmathbf}[1]{\ensuremath{\boldsymbol{#1}}}
\newcommand{\tmop}[1]{\ensuremath{\operatorname{#1}}}
\newcommand{\tmstrong}[1]{\textbf{#1}}
\newtheorem{theorem}{Theorem}
\newtheorem{corollary}[theorem]{Corollary}
\newtheorem{lemma}[theorem]{Lemma}
\newtheorem{remark}[theorem]{Remark}
\newcommand{\norm}[1]{\left\lVert#1\right\rVert}
\numberwithin{theorem}{section}

\begin{document}

\title{Isotopy Uniqueness of Self-diffeomorphisms of Handlebodies}
\author{Fang Sun and Xuezhi Zhao}
\maketitle
\begin{abstract}
The mapping class group $MCG(\Sigma_g)$ of a surface of genus $g$ has a long-history in topology and group theory. More recently, the mapping class group $MCG(V_g)$ of a handlebody $V_g$ of genus $g$ has become an interesting topic in the study of $3$ manifolds, largely thanks to Heegaard splitting. While $MCG(V_g)$  can be regarded naturally as a sub group of $MCG(\Sigma_g)$, we could not find any complete proof of this fundamental theorem. 
It is the purpose of this paper that we give a rigorous proof of embedding of $MCG(V_g)$ into $MCG(V_g)$. The key step is:  Any self-homeomorphism $f$ of handlebody $V_g$ of genus $g$ is ambient isotopic to identity  if the restriction $f|_{\partial V_g}$ is isotopic to identity.
\end{abstract}

\section{Introduction}
A handlebody $V_g$ of genus $g$ can be viewed as a 3-ball with $g$ 1-handles attached. 
Of great interest to topologists is the mapping class group $MCG(V_G)$, defined as the group of isotopy classes of orientation preserving self-homeomorphisms of $V_g$. Meanwhile. the mapping class group $MCG(\partial V_g)$ has been extensively studied (c.f. \cite{PrimerMCG}). The two groups are related by a restriction map $r: MCG(V_g) \rightarrow MCG(\partial V_g), r(f)=f_{| \partial V_g}$. If $r$ is injective, then knowledge about $MCG(\partial V_g)$ can be applied for the study of $MCG(V_g)$ (see \cite{PrimerHandlebody}). This injectivity is widely regarded as true, yet an explicit and detailed proof seems to be absent in the literature. A sketch of proof is given in \cite{PrimerHandlebody} Lemma 3.1, but it seems to have used technique exclusively from the smooth category (transversality) and the topological category (Alexander's trick) simultaneously. The main purpose of this paper is to provide a rigorous proof for the following result:

\begin{theorem} \label{MainTheoremTopo}
 Let $f$ be a self-homeomorphism of $V_g$ such that
  $f_{| \partial V_g}$ is isotopic to identity. Then f is ambient isotopic to identity.
\end{theorem}

Every self-homeomorphism on a compact 3-manifold with boundary is ambient isotopic to a diffeomorphism. This can be proved using local contractibility of the group of homeomorphisms for a compact manifold (see \cite{LocalContractHomeo}), together with the denseness of diffeomorphisms in that group for dimension three (\cite{ObstructionSmoothingHomeo} Theorem 6.3). Hence it suffice to work within the $C^\infty$ category, and prove the following:

\begin{theorem} \label{MainTheoremSmooth}
 Let $f$ be a self-diffeomorphism of $V_g$ such that
  $f_{| \partial V_g}$ is smoothly isotopic to identity. Then f is smoothly ambient isotopic to identity.
\end{theorem}

Theorem \ref{MainTheoremSmooth} will be a consequence of the following stronger result (see Section 4):
\begin{theorem} \label{Handlebody}
  Let $f$ be a self-diffeomorphism of $V_g$ such that
  $f_{| \partial V_g} = \tmop{id}$. Then f is ambient isotopic rel
  $\partial V_g$ to identity.
\end{theorem}

We will prove \ref{Handlebody} by induction. The induction starts with $g=1$, where the handlebody is just a solid torus. Because a solid torus is easy to parametrize, and because the proof of the induction step is almost parallel to the initial step, we shall prove the following special case in detail, and indicate the necessary modifications for the induction step.

\begin{theorem} \label{Main}
  Let $f$ be a self-diffeomorphism of $\mathcal{M}= D^2 \times S^1$ such that
  $f_{| \partial \mathcal{M}} = \tmop{id}$. Then f is ambient isotopic rel
  $\partial \mathcal{M}$ to identity.
\end{theorem}

The paper is organized as follows: In Section 2 we present some preliminaries that will be useful for our purpose. In Section 3 we give a proof of \ref{Main}. Finally, in Section 4 we finish the induction and prove Theorem \ref{Handlebody}, alongside \ref{MainTheoremSmooth}.

\section{Preliminaries}

A few preliminary results are needed before we start with the proof of the
theorem.

Throughout this paper, we denote the unit interval $[0, 1]$ by $I$.

\subsection{Review of Isotopy, Tubular Neighborhoods}

To avoid confusion on terminology, we review the definition of various notions
of isotopy in the following. In these definitions, we assume all spaces to be
smooth ($C^\infty$) manifolds and all maps to be smooth.

An {\tmem{isotopy}} is a map $F : X \times I \rightarrow Y$ such that the
{\tmem{track}} $\widehat{F} : X \times I \rightarrow Y \times I$ defined by
$\widehat{F} (x, t) = (F (x, t), t), x \in X, t \in I$ is an embedding. We also
say $F$ is an isotopy between $F_0$ and $F_1$, and $F_0$ and $F_1$ are
isotopic. If $A$ is a subspace of $X$ such that $F_t = F_0$ on $A$ for all $t
\in I$, then $F$ is said to be {\tmem{stable on}} $A$. We also call $F$ an
{\tmem{isotopy rel}} $A$. When $X$ is a smooth manifold with boundary, $X \times I$ is a smooth manifold with corner. Recall that a map $f : M \rightarrow N$ between smooth manifolds with corners is called smooth if for any (corner) charts $\phi :U \rightarrow M$, $ \varphi :V \rightarrow N$, $\varphi^{-1} \circ f \circ \phi$ admits an extension to a smooth map in an open (with respect to $\mathbb{R}^m$) neighborhood of each point, where $m= dim M$.

The {\tmem{support}} of an isotopy $F : X \times I \rightarrow Y$, denoted as
$\tmop{supp} F$, is the closure of $\{ x \in X|F (x, t) \neq F (x, 0)
\tmop{for} \tmop{some} t \in I \}$.

If $Z$ is a (embedded) submanifold of $Y$, an {\tmem{isotopy of}} $Z$
{\tmem{in}} $Y$ is an isotopy $F : Z \times I \rightarrow Y$ where $F_0$ is
the identity.

An {\tmem{ambient isotopy}} of Y is an isotopy $H : Y \times I \rightarrow Y$
such that $H_0 = \tmop{id}$ and the track $\hat{H} : Y \times I \rightarrow Y
\times I$ is a diffeomorphism. In particular, for each $t \in I$, $H_t$ is a
diffeomorphism of $Y$.

An isotopy $F : X \times I \rightarrow Y$ is {\tmem{ambient}} if there is an
ambient isotopy $H$ of $Y$ with $F = H \circ (F_0 \times \tmop{id}_I)$. In
this case we say $F_0$ and $F_1$ are ambient isotopic. If $F$ is ambient, so
is $\bar{F} : (x, t) \mapsto F (x, 1 - t)$.

The relation of two maps being isotopic is an equivalence relation: the only
nontrivial part of this claim is transitivity, for which one need to smooth
the joint of two consecutive isotopies (see for example \cite{Hirsch} p.111). The
same is true with ``isotopic'' replaced by ``ambient isotopic''.

Let $A$ be a subset of $Y$, it is easy to verify that a diffeomorphism $f : Y
\rightarrow Y$ is ambient isotopic rel $A$ to identity if and only if there is
an ambient isotopy of $Y$ ending with $f$ and stable on $A$.

Finally, recall that a {\tmem{time-dependent vector field}} on a smooth
manifold $Y$ is a smooth $G : Y \times I \rightarrow T Y$ such that $G (y, t)
\in T_y Y$, where $T Y$ is the tangent bundle of $Y$ and $T_y Y$ is the tangent
space at $y \in Y$. The {\tmem{support}} of $G$ (denoted as $\tmop{supp} G$)
is defined as the closure in $Y$ of $\{ y \in Y|G (y, t) \neq 0 \tmop{for}
\tmop{some} t \in I \}$. A compactly supported time dependent vector field $G$
on $Y$ generates a compactly supported ambient isotopy $F$ of $Y$, in the
sense that
\[ \frac{\partial F}{\partial t} (y, t) = G (F (y, t), t) \]
(c.f. \cite{Hirsch} Theorem 8.1.1)

We will be using the definition of tubular neighborhoods that appears in
\cite{Hirsch}. To be precise, let $N$ be a submanifold of $M$, a {\tmem{tubular
neighborhood}} of $N$ in $M$ is a smooth embedding of a vector bundle over $N$
into $M$ that is an open map and equals to the inclusion of $N$ when
restricted to the zero section.

\subsection{The Isotopy Extension Theorem}

We shall use the following Isotopy Extension Theorem (\cite{Hirsch} Theorem 8.1.3),
as well as its proof.

\begin{theorem} \label{IET}
  Let M be a smooth manifold without boundary. Let $N \subset M$ be a compact
  embedded submanifold without boundary, and $F : N \times I \rightarrow M$ be an
  isotopy of $N$. Then $F$ can be extended to an ambient isotopy of $M$ having
  compact support.
\end{theorem}

\begin{proof}
  See \cite{Hirsch} p.180.
\end{proof}

We will need to deal with the case where $N$ has non-empty boundary. The
original proof uses a tubular neighborhood of $\widehat{F} (N \times I)$ to extend
the vector field $X$. To avoid dealing with tubular neighborhoods of
submanifolds with corners, we use the following lemma as a way around.

\begin{lemma} \label{VF}
  Let $M$ be a smooth manifold without boundary and $N \subset M$ be a compact
  embedded submanifold with boundary. Let $F : N \times I \rightarrow M$ be an
  isotopy of $N$ in $M$, and let $\widehat{F}$ be the track of $F$. Define a
  vector field $X'$ on $\widehat{F} (N \times I)$ as the tangent vectors of curves
  $t \mapsto \widehat{F} (x, t), x \in N$. Then there exists a compactly supported
  smooth vector field $X$ on $M \times I$ that extends $X'$. 
\end{lemma}

\begin{proof}
  Embed $N \times I$ as a closed subset in $N \times \mathbb{R}$. One may choose charts of $N \times I$ of the form $W \times J$, where $W$ is a chart of $N$ and $J$ is either $(0,1]$ or $[0,1)$. Using such charts and the smoothness of $F$, we can extend $F$ smoothly in a neighborhood of each point of $N \times I$. By an argument of partition of unity (or use \cite{LeeSM} Corollary 6.27), we could find a smooth extension of $F$ defined on $N \times \mathbb{R}$. We still denote this extension by $F$.
  
  Since the collection of embeddings is open in $C^{\infty}(N,M)$ with respect to the Whitney
  topology and $N$ is compact, we can find $\varepsilon > 0$ such that $F_t$
  is an embedding for $t \in (- \varepsilon, 1 + \varepsilon)$. Using
  compactness of $N$ once again and shrinking $\varepsilon$ if necessary, we
  may assume $\widehat{F} : N \times (- \varepsilon, 1 + \varepsilon) \rightarrow
  M \times \mathbb{R}$ is an embedding.
  
  Define a vector field $X''$ on $\widehat{F} (N \times (- \varepsilon, 1 +
  \varepsilon))$ as the tangent vectors of curves $t \mapsto \widehat{F} (x, t), x
  \in N$. Then $X''$ extends $X'$. Since $\widehat{F} (N \times (- \varepsilon, 1
  + \varepsilon))$ is an embedded submanifold with boundary, we can use local
  slice charts (see, for example, \cite{LeeSM} Theorem 5.51) to extend $X''$ in a
  neighborhood of each point of $\widehat{F} (N \times (- \varepsilon, 1 +
  \varepsilon))$. Now it is easy to extend $X'$ smoothly to a vector field $X$
  defined on $M \times I$ by a partition of unity argument. Alternatively, we
  can apply \cite{LeeSM} Lemma 8.6 to $X'$ and $\widehat{F} (N \times I) \subset M
  \times \mathbb{R}$. By multiplication with a smooth function $\rho : M
  \times I \rightarrow I$ that equals to $1$ on $\widehat{F} (N \times I)$ and equals to $0$
  outside a compact neighborhood of $\widehat{F} (N \times I)$, we may further
  assume $X$ to be compactly supported.
\end{proof}

Substituting the above lemma for the tubular neighborhood argument in the proof
of \cite{Hirsch} Theorem 8.1.3, we have:

\begin{corollary} \label{CIET}
  The conclusion of Theorem \ref{IET} remains true if $\partial N \neq \varnothing$.
\end{corollary}

\subsection{The Isotopy Uniqueness of Collars}

We shall also use the following theorem of isotopy uniqueness of collars. This is
a well known result, yet it seems a clear proof can not be easily found in the
literature. For the sake of completeness, we prove a version that suit our
purpose.

\begin{theorem} \label{IUC}
  (Isotopy Uniqueness of Collars) Let $M$ be a smooth manifold with boundary.
  Let $C_i : \partial M \times [0, + \infty) \rightarrow M, i = 0, 1$ be
  smooth embeddings (collars). Then there is an isotopy $F : \partial M \times
  [0, + \infty) \times I \rightarrow M$ such that $F_0 = C_0$ and $F_1 = C_1$.
  If in addition $\partial M$ is compact and $M$ is connected, then for any $a
  > 0$, there is an ambient isotopy $G$ of $M$ rel $\partial M$ with $G_1
  \circ C_0 = C_1$ on $\partial M \times [0, a]$.
\end{theorem}

\begin{proof}
  Take the double $D M = M \cup_{\partial M} M'$ of $M$, where $M'$ is another
  copy of $M$. Define a smooth manifold structure on $D M$ via any collar of
  $\partial M$. We start with the general case.
  
  We wish to extend $C_i, i = 0, 1$ to tubular neighborhoods of $\partial M$
  in $D M$. For each $C_i, i = 0, 1$, define a vector field $V_i$ on
  $\tmop{Im} C_i$ by the tangent vectors of curves $t \mapsto C_i (x, t), x
  \in \partial M$. Extend $V_i$ to a smooth vector field $\widetilde{V_i}$ on
  $M' \cup \tmop{Im} C_i$ (using, say, \cite{LeeSM} Lemma 8.6). By the Flowout
  Theorem (\cite{LeeSM} Theorem 9.20), for each $i$ there exist a smooth positive function
  $\delta_i : \partial M \rightarrow \mathbb{R}$ such that the flow $\Phi_i$
  of $\widetilde{V_i}$, when restricted to $O_{\delta_i} = \{ (x, t) \in \partial M
  \times \mathbb{R}, \nobracket \norm{t} < \delta_i (x)  \} \nobracket$, is an
  embedding into $D M$. The uniqueness of integral curves implies that
  the domain of $\Phi_i$ includes $\partial M \times [0, + \infty)$, on which it equals to $C_i$.
  
  For any $x \in \partial M,i=0,1$, consider the curve $\Phi_{i, x} : t \mapsto
  \Phi_i (x, t)$. Since $\partial M$ separates $D M$, $\Phi_{i, x}$ maps $(-
  \delta_i (x), 0)$ into either $\tmop{Int} M$ or $\tmop{Int} M'$. In the
  former case, the tangent vector of $\Phi_{i, x}$ at $t = 0$ would be tangent
  to $\partial M$, which is impossible. Hence $\Phi$ maps $O_{\delta_i} \cap
  (\partial M \times (- \infty, 0])$ into $M'$, and $\Phi_i$ is an embedding
  when restricted to $O_{\delta_i}^+ = \{ (x, t) |t \in (- \delta_i (x), +
  \infty)  \}$. An easy rescaling of $\Phi_{i|O_{\delta i}^+}$ produces a
  smooth embedding $T_i : \partial M \times \mathbb{R} \rightarrow D M$ that
  coincides with $C_i$ on $\partial M \times [0, + \infty)$. Note that $T_i$ is
  a tubular neighborhood of $\partial M$ in $D M$.
  
  By the isotopy uniqueness of tubular neighborhoods (c.f. \cite{Hirsch} Theorem
  4.5.3), there is an isotopy $\widetilde{F} : \partial M \times \mathbb{R} \times
  I \rightarrow D M$ rel $\partial M \times \{ 0 \}$ such that $\widetilde{F}_0 =
  T_0, \widetilde{F}_1 (\partial M \times \mathbb{R}) = T_1 (\partial M \times
  \mathbb{R})$ and $T_1^{- 1} \circ \widetilde{F}_1$ is a smooth vector bundle
  isomorphism over $\partial M$. In other words, there is a smooth function
  $\rho : \partial M \rightarrow \mathbb{R}- \{ 0 \}$ such that $\widetilde{F}_1
  (x, t) = T_1 (x, \rho (x) t)$. The separability of $M$ by $\partial M$
  guarantees that $\tmop{Im} \rho \subset (0, + \infty)$. So we may define an
  isotopy $H$ rel $\partial M \times \{ 0 \}$ between $\widetilde{F}_1$ and $T_1$ by
  \[ H : \partial M \times \mathbb{R} \times I \rightarrow D M, H (x, s, t)
     \assign T_1 (x, (s + (1 - s) \rho (x)) t) \]
  Combining $\widetilde{F}$ and $H$ (and smoothing their joint), we obtain an isotopy
  $F$ between $T_0$ and $T_1$ rel $\partial M \times \{ 0 \}$. Now
  separability of $D M$ by $\partial M$ guarantees that the restriction of $F$
  to $\partial M \times [0, + \infty) \times I$ is the desired isotopy.
  
  For the second statement of the theorem, simply apply Corollary \ref{CIET} to the
  embedding $T_0 : \partial M \times [- a, a] \rightarrow M$ and the isotopy $F$.
\end{proof}

\begin{remark}
Note that in the above theorem, the interval $[0,+\infty)$ could be replaced by $[0,A)$ for any $A>0$, in which case $a$ is required to be within $(0,A)$.
\end{remark}

\subsection{The Schoenflies Theorems}

We will be using both the topological and smooth version of Schoenflies Theorem in dimension 2 and 3. In the smooth category we have (c.f. \cite{Hatcher3manifolds}):

\begin{theorem}
Suppose $h: S^{n-1} \rightarrow S^n$ ($n=2,3$) is a smooth embedding. Then there is a self diffeomorphism of $S^n$ taking $h(S^{n-1})$ to the equator.
\end{theorem}

In the topological case, assuming some extra regularity, the Schoenflies Theorem holds for all dimensions (c.f.\cite{BrownSchoenflies}):

\begin{theorem}
Suppose $h: S^{n-1} \times I \rightarrow S^n$ is an embedding. Then there is a self homeomorphism of $S^n$ taking $h(S^{n-1} \times \{1/2\})$ to the equator.
\end{theorem}
 
We will be dealing with topological or smooth hyper-spheres embedded in $D^2$ or $D^2 \times
S^1$. We could embed the ambient space smoothly as a submanifold of
the Euclidean space with the same dimension. If the conditions of the Schoenflies Theorem in the respective category are satisfied(in the smooth case this is automatic), then we can talk about the disk (ball) bounded within the sphere, and
it is independent of the embedding.

\subsection{Compressing 3-balls by Isotopy, and Rescaling Spheres Near
Their Equators}

The next two technical lemmas will be useful in the third step of our proof.

\begin{lemma} \label{Compress}
  Let $B^3$ be the unit 3-ball and $D \subset B^3$ be a smoothly, neatly
  embedded 2-disk. Let $\widetilde{D}$ be one of the two disks bounded by
  $\partial D$ in $\partial B^3$. Suppose $U$ is a tubular neighborhood of
  $\partial D$ in $\partial B^3$. Then there is a submanifold $B$ of $B^3$
  such that\\
  
  (i) $B$ is diffeomorphic to a ball
  
  (ii) $B \cap D = \varnothing$
  
  (iii) $\widetilde{D} - U \subset B$ 
  \begin{flushleft}
  and a smooth isotopy $H : B^3 \times I \rightarrow B^3$ such that $H_0 =
  \tmop{id}, H_1 (B^3) = B$ and $H$ is stable one a neighborhood of $\widetilde{D}
  - U$ in $B^3$.
  \end{flushleft}
\end{lemma}

\begin{proof}
  We claim that the problem can be reduced to the case where $\partial D =
  \partial B^3 \cap (\mathbb{R}^2 \times \{ 0 \})$. Since any two smooth knots
  in $S^2$ are smoothly ambient isotopic, there is an ambient isotopy $G$ of
  $S^2$ such that $G_1 (\partial D) = S^2 \cap (\mathbb{R}^2 \times \{ 0 \})$.
  The diffeomorphism $G_1$ of $S^2$ can be extended to a self diffeomorphism
  of $B^3$ (c.f. \cite{Smale2sphere}). This proves the
  claim.
  
  We could further assume $\widetilde{D}$ is the upper hemisphere. There exist
  $\varepsilon > 0$ such that $\partial B^3 \cap (\mathbb{R}^2 \times (-
  \varepsilon, \varepsilon)) \subset U$. On the other hand, we could choose a
  radial collar $C$ of $\partial B^3 \cap \left( \mathbb{R}^2 \times \left( -
  \infty, - \frac{\varepsilon}{2} \right) \right)$ that is disjoint from $D$.
  Choose an isotopy $H : B^3 \times I \rightarrow B^3$ such that $H_0 =
  \tmop{id}, H_1 (B^3) \subset C$ and $H$ is stable on a neighborhood of
  $\partial B^3 \cap (\mathbb{R}^2 \times (- \infty, \varepsilon])$.
  Intuitively, we are compressing $B^3$ by pushing it from the top. Now $B =
  H_1 (B^3)$ has the desired properties.
\end{proof}

\begin{lemma} \label{rescale}
  Let $X : S^{n - 1} \rightarrow (0, + \infty)$ be a smooth function. Define a
  embedding $N \of S^{n - 1} \times (- 1, 1) \rightarrow S^n$ by
  \[N (z, t) = \left(
     \sqrt{1 - t^2} \frac{z}{\norm{z}}, t \right), z \in S^{n - 1}, t \in (- 1,
     1) \]
  where $S^{n - 1}$ (resp. $S^n$) are regarded as subspaces of $\mathbb{R}^n$
  (resp. $\mathbb{R}^{n + 1}$) and $\norm{\cdot}$ is the Euclidean norm. Then
  there is a diffeomorphism $\varphi : S^n \rightarrow S^n$ such that $\varphi
  (D^n_+) = D^n_+, \varphi (\tmop{Im} N) = \tmop{Im} N$ and $\varphi (N (z,
  t)) = N (z, X (z) t)$ for $(z, t)$ in a neighborhood of $S^{n - 1} \times \{
  0 \}$. Here $D_+^n$ denotes the upper hemisphere in $S^n$.
\end{lemma}

\begin{proof}
  Define $M \assign \underset{z \in S^{n - 1}}{\max} X (z)$. Choose a smooth
  function $\alpha : I \rightarrow I$ with the following properties:
  
  (i) $\alpha (t) = M t$ for $t \in \left[ 0, \frac{1}{3} \right)$
  
  (ii) $\alpha (t) = t$ for $t \in \left( \frac{2}{3}, 1 \right]$
  
  (iii) $\alpha$ is strictly increasing
  
  Choose another smooth function $\beta : I \rightarrow I$ such that
  
  (i) $\beta (t) = 1$ for $t \in \left[ 0, \frac{1}{3} \right)$
  
  (ii) $\beta (t) = 0$ for $t \in \left. (\nobracket \frac{2}{3}, 1 \right]$
  
  (iii) $\beta$ is non-increasing
  
  It is not hard to check, using monotonicity, that the function $\varphi_+ : S^{n - 1} \times [0,
  1) \rightarrow S^{n - 1} \times [0, 1)$ defined by
  \[ \varphi_+ (z, t) = \left( z,  \left( \frac{\beta (t) X (z)}{M}
     + 1 - \beta (t) \right) \alpha (t) \right) \]
  is a diffeomorphism. By our choice, $\varphi_+ (z, t) =
  (z, X (z) t)$ for $t \in \left[ 0, \frac{1}{3} \right)$ and $\varphi_+ (z,
  t) = (z, t)$ for $t \in \left( \frac{2}{3}, 1 \right)$.
  
  Similarly,we can define a diffeomorphism $\varphi_- : S^{n - 1} \times (-
  1, 0] \rightarrow S^{n - 1} \times (- 1, 0]$ with $\varphi_- (z, t) = (z, X
  (z) t)$ for $t \in \left( - \frac{1}{3}, 0 \right]$ and $\varphi_- (z, t) =
  (z, t)$ for $t \in \left( - 1, - \frac{2}{3} \right)$.
  
  Now if we patch up $\varphi_+$ and $\varphi_-$ while identifying $S^{n - 1}
  \times (- 1, 1)$ with its image under $N$, we obtain a self diffeomorphism
  of $S^n - \{ \pm P \}$, where $\pm P$ stand for the north and south poles.
  Extending this diffeomorphism to the poles by identity, we get a bijection
  $\varphi : S^n \rightarrow S^n$. The properties of $\varphi_+, \varphi_-$
  guarantees that $\varphi$ is a diffeomorphism with the desired properties.
\end{proof}

\subsection{Partial Extensions of Partial Collars}

The next lemma says a ``partial collar'' of the boundary of a ball can be
partially extended to an actual collar. This will be used in the fourth step.

\begin{lemma} \label{partialextension}
  Let $B^n$ be the unit ball in $\mathbb{R}^{n+1}$. For any interval $J$, define
  $E_J \assign \partial B^n \cap (\mathbb{R}^n \times J)$. Given $d > 0$ and a
  smooth embedding $C : E_{[- 1, 0]} \times [0, d] \rightarrow B^n$ such that
  $c (x, 0) = x$ for any $x \in E_{[- 1, 0]}$ and $C (E_{[- 1, 0]} \times (0,
  d]) \subset \tmop{Int} B^n$. Then for any $\varepsilon > 0$, there is
  $\delta \in (0, d]$ and a collar $C' \of \partial B^n \times [0, \delta]
  \rightarrow B^n$ that coincide with $C$ on $E_{[- 1, - \varepsilon]} \times
  [0, \delta]$.
\end{lemma}

\begin{proof}
  Choose a smooth collar $C_0:S^{n-1} \times [0,1) \rightarrow B^n$. Let $V$ (resp. $V_0$) be the vector field on $\tmop{Im} C$ (resp. $\tmop{Im} C_0$)
  defined as the tangent vectors of $t \mapsto C (x, t)$ (resp. $t \mapsto C_0 (x, t)$) for $x \in E_{[- 1, 0]}$ (resp. $x \in S^{n-1}$).
  With the help of the Invariance of Domain, it is not hard to see that $C
  (E_{[- 1, 0)} \times [0, d))$ is open in $B^n$. Let $\rho : B^n \rightarrow I$ be a smooth function such that:
  
  (i) $\rho (x) = 1$ if $x \in C (E_{[- 1, - \varepsilon]} \times [0,
  \frac{d}{2}])$
  
  (ii) $\tmop{supp} \rho \subset C (E_{[- 1, 0)} \times [0, d))$
  
  Define a smooth vector field $V'$ on $\tmop{Im} C_0$ by $V' \assign \rho V + (1 - \rho)
  V_0$.
  Both $V$ and $V_0$ are inward pointing on
  $\partial B^n$ (wherever defined). Hence the same is true for $V'$. The desired collar
  exists by the Boundary Flowout Theorem (\cite{LeeSM} Theorem 9.24).
\end{proof}

\subsection{Review of Fiber Derivative}

In order to perform the fifth step, let us briefly review the notion of fiber
derivative. Let $E_0, E_1$ be $C^{\infty}$ vector bundles over a smooth
manifold $M$ and $g : E_0 \rightarrow E_1$ be a smooth map that is identical
on the zero section ($g$ is not necessarily fiber preserving). Consider the
long exact sequences of vector bundles:
\[ 0 \rightarrow E_i \overset{\alpha_i}{\rightarrow} T E_{i|M}
   \overset{\beta_i}{\rightarrow} T M \rightarrow 0, i = 0, 1 \]
where $\alpha_i$ is defined by identifying the vector space $(E_i)_x$ with
$T_x (E_i)_x$ for each $x \in M$ while $\beta_i$ is induced by the projection
$p_i : E_i \rightarrow M$. For each $i$, the sequence splits by the map $T : T M \rightarrow
T E_{i|M}$ induced by the zero section. Hence we can canonically identify $T
E_{i|M}$ with $E_i \oplus T M$.

The {\tmem{fiber derivative}} of $g$ is the vector bundle homomorphism $\Phi :
E_0 \rightarrow E_1$ defined as the component along $E_i$'s of $T g : T E_{0|
M} \rightarrow T E_{1| M}$ with respect to the above direct sum decomposition.

\subsection{A Version of the Smale Conjecture and Smooth Paths in
$\tmop{Diff}_{\partial, U} (M)$}

The next few theorems will be used in the sixth step.

We will need the following version of the Smale Conjecture (Smale Theorem).

\begin{theorem} \label{SC}
  For any diffeomorphism $f : D^3 \rightarrow D^3$ such that $f_{| \partial
  D^3} = \tmop{id}$, there is a smooth ambient isotopy $F
  : D^3 \times I \rightarrow D^3$ rel $\partial D^3$ with $F_1 = f$.
\end{theorem}

\begin{proof}
  Regard $f$ as an embedding $D^3 \hookrightarrow \mathbb{R}^3$. Since $f_{|
  \partial D^3} = \tmop{id}$, $f$ preserves orientation. By \cite{Hirsch} Theorem
  8.3.1, there is a smooth isotopy $H : D^3 \times I \rightarrow \mathbb{R}^3$
  such that $H_0 = \tmop{id}$ (the standard inclusion) and $H_1 = f$. Let
  $\rho : I \rightarrow I$ be a smooth function such that for some
  $\varepsilon > 0$, $\rho (t) = 0$ for $t < \varepsilon$ and $\rho (t) = 1$
  for $t > 1 - \varepsilon$. Modifying $H$ by pre-composing with $\tmop{id}_{D^3}
  \times \rho$, we may assume $H_t = \tmop{id}$ for $t < \varepsilon$ and $H_t
  = f$ for $t > 1 - \varepsilon$. In particular $H (x, t) = x$ for $x \in
  \partial D^3, t \in [0, \varepsilon) \cup (1 - \varepsilon, 1]$.
  
  Define $h : S^2 \times I \rightarrow \mathbb{R}^3$ by $h = H_{| \partial D^3
  \times I}$, and define $p : I \rightarrow S^1$ by $p (t) = (\cos (2\pi t), \sin (2\pi
  t))$. There exist a smooth map $g : S^2 \times S^1 \rightarrow \mathbb{R}^3$
  such that $g \circ (\tmop{id}_{S^2} \times p) = h : S^2 \times I \rightarrow
  \mathbb{R}^3$. By the Theorem on \cite{HatcherSC} p.553, there is a smooth extension
  $\bar{G} : D^3 \times S^1 \rightarrow \mathbb{R}^3$ of $g$ such that for any
  $z \in S^1$, $\bar{G}_{|D^3 \times \{ z \}}$ is an embedding.
  
  Define $G : D^3 \times I \rightarrow \mathbb{R}^3$ by $G = \bar{G} \circ
  (\tmop{id}_{D^3} \times p)$. For each $t \in I$, both $G_t$ and $H_t$ are
  embeddings of $D^3$ into $\mathbb{R}^3$. Since $G_{t| \partial D^3} = H_{t|
  \partial D^3}$, we must have $G_t (D^3) = H_t (D^3)$. In particular, if we
  denote by $\widehat{G}, \widehat{H}$ the track of $G$, $H$, then $\tmop{Im} \widehat{G}
  = \tmop{Im} \widehat{H}$. By definition, we know $G_0 = G_1$. We also have
  $G_{0| \partial D^3} = h_0 = H_{0| \partial D^3} = \tmop{id}$, which implies
  $G_0 (D^3) = D^3$.
  
  Define $F : D^3 \times I \rightarrow D^3$ by $F = G_0 \circ \pi_1 \circ
  \widehat{G}^{- 1} \circ \widehat{H}$, where $\pi_1 : D^3 \times I \rightarrow D^3$
  is the natural projection. Then $F_t = G_0 \circ G_t^{- 1} \circ H_t$. We
  know $F_0 = H_0 = \tmop{id}, F_1 = H_1 = f$ (for $G_0 = G_1$), and $F_{t|
  \partial D^3} = G_0 \circ G_t^{- 1} \circ H_{t| \partial D^3} = G_{0|
  \partial D^3} = \tmop{id}$ (for $G_{t| \partial D^3} = H_{t| \partial D^3}$
  and $G_{0| \partial D^3} = \tmop{id}$). Hence $F$ is the desired isotopy.
\end{proof}

\begin{remark}
  The above theorem is not equivalent to the well known result that the space
  $\tmop{Diff}_{\partial} (D^3) = \{ \tmop{diffeomorphism} \tmop{of} D^3
  \tmop{that} \tmop{are} \tmop{identical} \tmop{on} \partial D^3 \}$ is path
  connected. Let $X$ be a smooth manifold. A continuous isotopy $F : X \times
  I \rightarrow Y$ where each $F_t$ is a diffeomorphism for all $t \in I$ and
  partial derivatives of $F_t$ of all orders varies continuously with respect
  to $t$ need not be smooth. A counter example can be given as: $F :
  \mathbb{R} \times I \rightarrow \mathbb{R}, F (x, t) = x + \max \left\{ 0, t
  - \frac{1}{2} \right\}$ (this example is taken from \cite{Kupers} p.55).
\end{remark}

Let $M$ be a smooth manifold with boundary. Denote by $\tmop{Diff} (M)$ the
collection of all self-diffeomorphisms. Define
\[ \tmop{Diff}_{\partial} (M) = \{ f \in \tmop{Diff} (M) |f_{| \partial M} =
   \tmop{id} \} , \]
\begin{align*}
\tmop{Diff}_{\partial, D} (M) = \{ f \in \tmop{Diff}_{\partial} (M) | &
   \tmop{partial} \tmop{derivatives} \tmop{of} f \tmop{of} \tmop{all} \tmop{order}\\
    & \tmop{coincide}
   \tmop{with} \tmop{those} \tmop{of} \tmop{id} \tmop{along} \partial M \},
\end{align*}
\[ \tmop{Diff}_{\partial, U} (M) = \{ f \in \tmop{Diff} (M) |f = \tmop{id}
   \tmop{on} \tmop{some} \tmop{neighborhood} \tmop{of} \partial M \} \]
Proposition 4.3.1 and Theorem 5.3.1 of \cite{Kupers} shows that the inclusions
$\tmop{Diff}_{\partial, D} (M) \hookrightarrow \tmop{Diff}_{\partial} (M)$ and
$\tmop{Diff}_{\partial, U} (M) \hookrightarrow \tmop{Diff}_{\partial, D} (M)$
are weak homotopy equivalences. Hence if $f \in \tmop{Diff}_{\partial,
U} (M)$, and there is an ambient isotopy $F$ of $M$ rel $\partial M$ such that $F_1 =
f$, then there exist a topological ambient isotopy $F'$ of $M$ such that $F'_1
= f$, $F_t \in \tmop{Diff}_{\partial, U} (M)$ for all $t \in I$ and partial
derivatives of $F_t$ varies continuously with respect to $t$. As we have seen
in the remark above, this is weaker than saying $F'$ is smooth. The following
two theorems show that this is not an issue.

\begin{theorem} \label{PathI}
  Let $M$ be a smooth manifold with boundary $\partial M$. If $f \in
  \tmop{Diff}_{\partial, U} (M)$ and there is an ambient isotopy $H$ of $M$ such that
  $H_1 = f$ and $H_t \in \tmop{Diff}_{\partial, D} (M)$ for all $t \in I$.
  Then there exist a smooth ambient isotopy $H'$ of $M$ such that $H'_1 = f$
  and $H'_t \in \tmop{Diff}_{\partial, U} (M)$.
\end{theorem}

\begin{proof}
  Choose a collar $C : \partial M \times [0, 1) \rightarrow M$ of $\partial M$
  and glue a collar $\partial M \times (- 1, 0]$ to $M$ via the identity map
  of $\partial M$. Define $\widetilde{M} = (\partial M \times (- 1, 0])
  \cup_{\partial M} M$. Then $\widetilde{M}$ has a natural $C^{\infty}$ structure.
  Let $\widetilde{C} : \partial M \times (- 1, 1) \rightarrow M$ be defined by
  combining $C$ and the natural inclusion of $\partial M \times (- 1, 0]$. By
  our assumptions, $f$ extends by identity to a diffeomorphism of $\widetilde{M}$.
  Similarly, $H$ extends by identity to an ambient isotopy of $\widetilde{M}$. We
  still call the extensions $f$ and $H$, respectively.
  
  Let $\eta : (- 1, 1) \rightarrow (- 1, 1)$ be a diffeomorphism such that
  
  (i) $\eta (s) = s$ for $s > \frac{2}{3}$
  
  (ii) $\eta (s) \geqslant s$ for all $s$
  
  (iii) $\eta ([0, 1)) = \left[ \frac{1}{2}, 1 \right)$
  
  Define a diffeomorphism $s_{\eta} : \widetilde{M} \rightarrow \widetilde{M}$ by
  \begin{eqnarray*}
    s_{\eta} (x) = \left\{\begin{array}{l}
      \widetilde{C} (y, \eta (s)), \tmop{if} x = \widetilde{C} (y, s)\\
      x \hspace{4em}, \tmop{otherwise}
    \end{array}\right. &  & 
  \end{eqnarray*}
  and define $g = s_{\eta} \circ f \circ s_{\eta}^{- 1}$. Let $G : \widetilde{M}
  \times I \rightarrow \widetilde{M}$ be defined by
  \[ G (x, t) = s_{\eta} \circ H (s_{\eta}^{- 1} (x), t) \]
  By our choice of $s_{\eta}$, for each $t \in I$, $G_t = s_{\eta} \circ H_t
  \circ s_{\eta}^{- 1}$ is a diffeomorphism of $\widetilde{M}$ that coincide with
  $\tmop{identity}$ on $\widetilde{C} \left( \partial M \times \left( - 1,
  \frac{1}{2} \right] \right)$. Hence $G_{|M \times I} : M \times I
  \rightarrow M$ is a smooth ambient isotopy ending with $g$ that is stable on
  $\widetilde{C} \left( \partial M \times \left[ 0, \frac{1}{2} \right] \right)$.
  In particular, $G_{t|M} \in \tmop{Diff}_{\partial, U} (M)$.
  
  Next we join $f$ and $g$ by an isotopy.
  
  Define $\psi : \widetilde{M} \times I \rightarrow \widetilde{M}$ by
  \[ \psi (x, t) = \left\{\begin{array}{l}
       \widetilde{C} (y, (1 - t) \eta (s) + t s), \tmop{if} x = \widetilde{C} (y, s)\\
       x \hspace{9em}, \tmop{otherwise}
     \end{array}\right. \]
  and let $\widehat{\psi}$ be the track of $\psi$. It is not hard to check that
  $\psi_0 = s_{\eta}, \psi_1 = \tmop{id}$ and $\psi_t$ is a diffeomorphism of
  $\widetilde{M}$ for all $t \in I$. In particular, $\widehat{\psi}$ is a
  diffeomorphism.
  
  The promised isotopy $F : \widetilde{M} \times I \rightarrow \widetilde{M}$ is
  defined by
  \[ F (x, t) = \psi (f \circ \widehat{\psi}^{- 1} (x, t), t) = \psi_t \circ f
     \circ \psi_t^{- 1} (x) \]
  Then $F$ is a smooth isotopy with $F_0 = g, F_1 = f$. Since $f \in
  \tmop{Diff}_{\partial, U} (M)$, there exist a smooth, positive function
  $\delta : \partial M \rightarrow \mathbb{R}$ such that $f = \tmop{id}$ on $W
  = \{ \widetilde{C} (x, s) | - 1 < s \leqslant \delta (x) \}$. By condition (ii)
  of $\eta$ we know that $(1 - t) \eta (s) + t s \geqslant s$. Hence
  $\psi_t^{- 1} (W) \subset W$ for all $t \in I$. Restricting $F$ to $M \times
  I$, we obtain an ambient isotopy of $M$ from $g$ to $f$ that is stable on
  $W$. Now combine $F_{|M \times I}$ with $G_{|M \times I}$ and smooth their
  joint.
\end{proof}

\begin{theorem} \label{PathII}
  Let $M$ be a smooth manifold with compact boundary $\partial M$. If
  $f \in \tmop{Diff}_{\partial, D} (M)$ and there is an ambient isotopy $H$ of $M$ rel
  $\partial M$ such that $H_1 = f$. Then there exist a smooth ambient isotopy
  $H'$ of $M$ such that $H'_1 = f$ and $H'_t \in \tmop{Diff}_{\partial, D}
  (M)$.
\end{theorem}

\begin{proof}
  Choose a collar $C : \partial M \times I \hookrightarrow M$. Consider the
  following short exact sequence of vector bundles over $\partial M$:
  \[ 0 \rightarrow \partial M \times \mathbb{R} \overset{\alpha}{\rightarrow}
     T M_{| \partial M} \overset{\beta}{\rightarrow} T \partial M \rightarrow
     0 \]
  where $\alpha (x, s) = s \frac{\partial}{\partial t} C (x, t)_{|t = 0}$ and
  $\beta$ is the natural projection. The sequence splits by the natural
  inclusion $T \partial M \hookrightarrow T M_{| \partial M}$. Hence we have a
  decomposition $T M_{| \partial M} = T \partial M \oplus \varepsilon$, where
  $\varepsilon = \partial M \times \mathbb{R}$ is the trivial bundle.
  
  Given any $g \in \tmop{Diff}_{\partial} (M)$, there is an induced bundle
  automorphism $T g$ of $T M_{| \partial M}$. Let $\chi (g) : \varepsilon
  \rightarrow T \partial M, \phi (g) : \varepsilon \rightarrow \varepsilon$ be
  components of $T g$. The other two components of $T g$ are $\tmop{id} : T
  \partial M \rightarrow T \partial M$ and $0 : T \partial M \rightarrow
  \varepsilon$. Choosing a global frame of $\varepsilon$ (say,
  $\frac{\partial}{\partial s}$), we may regard $\chi (g)$ as a vector filed
  on $\partial M$ and $\phi (g)$ as a smooth function on $\partial M$. It is
  obvious that $\phi (g)$ is positive on $\partial M$.
  
  Since $f \in \tmop{Diff}_{\partial, D} (M)$, we have $T f = \tmop{id}$, thus
  $\chi (f) = 0$ and $\phi (f) \equiv 1$.
  
  The first step is to replace $H$ by an ambient isotopy $F$ rel $\partial M$
  with $F_1 = f$ and $\phi (F_s) \equiv 1$ for all $s \in I$.
  
  Choose a smooth $\eta : (0, + \infty) \times I \rightarrow I$ such that
  
  (i) for each $l \in (0, + \infty)$, $\eta_l : I \rightarrow I$ is a
  diffeomorphism that maps 0 to 0 and $\eta_l = \tmop{id}$ on $\left[
  \frac{1}{2}, 1 \right]$
  
  (ii) for each $l \in (0, + \infty)$, the derivative of $\eta_l$ at 0 is
  $\frac{1}{l}$
  
  (iii) $\eta_1 = \tmop{id}_I$
  
  For any smooth positive function $\lambda$ on $\partial M$, define
  $\widetilde{\eta} (\lambda) : M \rightarrow M$ by
  \[ \widetilde{\eta} (\lambda) (p) \assign \left\{\begin{array}{l}
       C (q, \eta (\lambda (q), t)), \tmop{if} p = C (q, t)\\
       p \hspace{6em}, \tmop{otherwise}
     \end{array}\right. \]
  It is not hard to see that $\widetilde{\eta} (\lambda) \in
  \tmop{Diff}_{\partial} (M)$, and that $\phi (\widetilde{\eta} (\lambda)) =
  \frac{1}{\lambda} : \varepsilon \rightarrow \varepsilon$. Hence for any $g
  \in \tmop{Diff}_{\partial} (M)$, we have $\phi (\widetilde{\eta} (\phi (g))
  \circ g) \equiv 1$.
  
  Now define $F : M \times I \rightarrow M$ by
  \[ F (p, s) : = H(\widetilde{\eta} (\phi (H_s))(p),s) =
     \left\{\begin{array}{l}
       H(C (q, \eta (\phi (H_s) (q), t)),s), \tmop{if} p = C (q, t)\\
       H (p, s) \hspace{8em}, \tmop{otherwise}
     \end{array}\right. \]
  For each $s \in I$, $F_s = H_s \circ \widetilde{\eta} (\phi (H_s)) \in
  \tmop{Diff}_{\partial} (M)$ and $\phi (F_s) \equiv 1$. To see that $F$ is
  smooth, note that $\phi (H_s)$ is determined by the partial derivatives of
  $H_s$, hence $\phi (H_s) (q)$ depends smoothly on $(q, s)$. We can compute
  $F_0 = H_0 \circ \widetilde{\eta} (\phi (H_0)) = \widetilde{\eta} (\phi (\tmop{id}))
  = \widetilde{\eta} (c_1) = \tmop{id}$, where $c_1$ is the constant function with
  value 1, and $F_1 = H_1 \circ \widetilde{\eta} (\phi (H_1))  = f  \circ \widetilde{\eta} (\phi
  (f)) = f \circ \widetilde{\eta} (c_1) = f$. It is obvious that $F$ is
  stable on $\partial M$.
  
  The second step is to construct the desired $H'$ from $F$.
  
  Choose a smooth function $\rho : I \rightarrow I$ such that $\rho (t) = t$
  for $t < \frac{1}{3}$ and $\rho (t) = 0$ for $t > \frac{2}{3}$.
  
  For any vector field $X$ on $\partial M$, let $\varphi_X : \partial M \times
  \mathbb{R} \rightarrow \partial M$ be the flow of $X$ (this is where we need
  compactness of $\partial M$). Define $\Phi (X) : M \rightarrow M$ by
  \[ \Phi (X) (p) = \left\{\begin{array}{l}
       C (\varphi_X (q, t \rho (t)), t), \tmop{if} p = C (q, t)\\
       p \hspace{6em}, \tmop{otherwise}
     \end{array}\right. \]
  It is easy to check that $\Phi (X)$ is well-defined and smooth, and that
  $\Phi (X)_{| \partial M} = \tmop{id}$. We also have $\Phi (X) \circ \Phi (-
  X) = \tmop{id}$, hence $\Phi (X) \in \tmop{Diff}_{\partial} (M)$. For $t <
  \frac{1}{3}$, $\Phi (X) \circ C (q, t) = C (\varphi_X (q, t), t)$, thus
  $\chi (\Phi (X)) = X, \phi (\Phi (X)) \equiv 1$.
  
  Now define $H' : M \times I \rightarrow M$ by
  \[ H' (p, s) = \Phi (- \chi (F_s)) \circ F (p, s) \]
  For each $s \in I$, it is easy to verify that $\Phi (- \chi (F_s)) \circ F_s
  \in \tmop{Diff}_{\partial, D} (M)$. We may compute $H'_0 = \Phi (- \chi
  (F_0)) \circ F_0 = \Phi (- \chi (\tmop{id})) = \Phi (0) = \tmop{id}$ and
  $H'_1 = \Phi (- \chi (F_1)) \circ F_1 = \Phi (- \chi (f)) \circ f = \Phi (0)
  \circ f = f$. It remains to verify the smoothness of $H'$. Since $F$ is
  smooth, it suffice to check the smoothness of the map $G : M \times I
  \rightarrow M$
  \[ G (x, s) = \Phi (- \chi (F_s)) (x) = \left\{\begin{array}{l}
       C (\varphi_{- \chi (F_s)} (q, t \rho (t)), t), \tmop{if} x = C (q, t)\\
       x \hspace{9em}, \tmop{otherwise}
     \end{array}\right. \]
  It boils down to checking that $\varphi_{- \chi (F_s)} (q, t \rho (t))$ depends smoothly on $(s,q,t)$. This is taken care of by the next lemma.
  
  \ 
\end{proof}

\begin{lemma}
  Let $M$ be a manifold without boundary, J be an interval, $V : M \times J
  \rightarrow T M$ be a time dependent vector field. For each $\mu \in J$, let
  $V_{\mu} : M \rightarrow T M$ be the restriction of $V$ on $M \times \{ \mu
  \}$, and denote the flow generated by $V_{\mu}$ as $\varphi_{\mu}$. Then the
  map $F$ sending $(x, t, \mu) \in M \times \mathbb{R} \times J$ to
  $\varphi_{\mu} (x, t)$ is defined on an open subset $O$ of $M \times
  \mathbb{R} \times J$, and is smooth on $O$.
\end{lemma}

\begin{proof}
  This is a local problem. Choose $(x_0, t_0, \mu_0) \in O$. It suffices to
  check that $F$ is defined and smooth in a neighborhood of $(x_0, t_0,
  \mu_0)$.
  
  First, assume in addition that there is a chart $(U, \psi)$ of $M$ such that
  $\varphi_{\mu_0} (x_0, t) \in U$ for all $t$ between $0$ and $t_0$. In this
  case. the assertion follows from the smooth dependence of solutions of
  ordinary differential equations on parameters. For example, we can apply \cite{ODECoddington}
  Theorem 7.4 (p.29), Theorem 7.5 and the remarks that follows.
  
  Now we consider the general case. Without loss of generality, let us assume
  $t_0 \geqslant 0$. There are charts $\{ (U_i, \psi_i), 1 \leqslant i
  \leqslant l \}$ of $M$ for some $l$ and a subdivision $0 = s_0 < s_1 < s_2 <
  \cdots < s_l = t_0$ of $[0, t_0]$ such that $\varphi_{\mu_0} (x_0, t) \in
  U_i$ for $t \in [s_{i - 1} .s_i], 1 \leqslant i \leqslant l$. By the above
  special case, $\varphi_{\mu} (x, s_1)$ is defined and smooth on an open
  neighborhood of $(x_0, \mu_0)$ in $M \times J$. Suppose the same can be said
  about $\varphi_{\mu} (x, s_{i - 1})$ for some $i$. The property of flows
  guarantees that $\varphi_{\mu} (x, t) = \varphi_{\mu} (\varphi_{\mu} (x,
  s_{i - 1}), t - s_{i - 1})$ wherever the two sides are defined. Applying the
  above special case, we conclude that $\varphi_{\mu} (x, s_i)$ is defined if
  $(\varphi_{\mu} (x, s_{i - 1}), \mu)$ is in certain neighborhood of
  $(\varphi_{\mu_0} (x_0, s_{i - 1}), \mu_0)$, and $\varphi_{\mu} (x, s_i)$
  depends smoothly on $(\varphi_{\mu} (x, s_{i - 1}), \mu)$. Thus
  $\varphi_{\mu} (x, s_i)$ is defined and smooth on an open neighborhood of
  $(x_0, \mu_0)$ in $M \times J$. By induction, we see $\varphi_{\mu} (x,
  t_0)$ has this property as well. Finally, the equation $\varphi_{\mu} (x, t)
  = \varphi_{\mu} (\varphi_{\mu} (x, t_0), t - t_0)$ and the conclusion of the
  special case implies $\varphi_{\mu} (x, t)$ is defined and smooth in a
  neighborhood of $(x_0, t_0, \mu_0)$.
\end{proof}

\section{Proof of Theorem \ref{Main}}

The proof is divided into six steps. In each step, we show the diffeomorphism obtained at the end of the previous step can be isotoped to one with stronger properties.

\subsection{The First Step}

\begin{theorem} \label{1st}
  Let $f$ be a self-diffeomorphism of $\mathcal{M}= D^2 \times S^1$ such that
  $f_{| \partial \mathcal{M}} = \tmop{id}$. Then $f$ is ambient isotopic rel
  $\partial \mathcal{M}$ to a diffeomorphism $f'$ of $\mathcal{M}$ such that
  there exist $0<r_0<1$ with
  \begin{equation}
    f' ((1 - r) z_1, z_2) = ((1 - r) z_1, e^{r i} z_2), 0 \leqslant r
    \leqslant r_0
  \end{equation}
  for any $z_1 \in \partial D^2, z_2 \in S^1$. Here we are regarding $D^2$ and
  $S^1$ as subsets of $\mathbb{C}$, and multiplications in (1) are that of
  complex numbers.
\end{theorem}

\begin{proof}
  Define two collars of $\partial \mathcal{M}$ in $\mathcal{M}$ as:
  \[ \Phi : \partial \mathcal{M} \times [0, 1) \rightarrow \mathcal{M}, \Phi
     (z_1, z_2, r) = ((1 - r) z_1, z_2), z_1 \in \partial D^2, z_2 \in S^1, r
     \in [0, 1) \]
  \[ \Psi : \partial \mathcal{M} \times [0, 1) \rightarrow \mathcal{M}, \Psi
     (z_1, z_2, r) = ((1 - r) z_1, e^{r i} z_2), z_1 \in \partial D^2, z_2 \in
     S^1, r \in [0, 1) \]
  Applying Theorem \ref{IUC} to the collars $f \circ \Phi$ and $\Psi$ with $a =
  \frac{1}{2}$, we obtain an ambient isotopy $G$ of $\mathcal{M}$ such that
  $G_1 \circ f \circ \Phi = \Psi$ on $\partial \mathcal{M} \times \left[ 0,
  \frac{1}{2} \right]$. This translates into
 \begin{align}
     &  G_1 \circ f ((1 - r) z_1, z_2) & \nonumber \\
    =& G_1 \circ f \circ \Phi (z_1, z_2, r) & \nonumber\\
    =  &\Psi (z_1, z_2, r) & \nonumber \\
    = & ((1 - r) z_1, e^{r i} z_2) & \nonumber 
 \end{align}
for $r \in \left[ 0, \frac{1}{2} \right]$. Hence $f' = G_1 \circ f$ is the
  desired map.
\end{proof}

\subsection{The Second Step}

For notational convenience, throughout the rest of our paper, we define $P_0 =
 1 \in S^1 \subset \mathbb{C}$, and for any interval $J$, define $A_J
= \{ x \in D^2 | \norm{x} \in J \}$.

Let $r$ be either a non-negative integer or $\infty$, $X, Y$ be $C^r$
manifolds. Denote by $C_s^r (X, Y)$ the collection of $C^r$ maps from $X$ to
$Y$ endowed with the strong Whitney topology (c.f. \cite{Hirsch} p.35).

\begin{theorem} \label{2nd}
  Let $f$ be a self-diffeomorphism of $\mathcal{M}= D^2 \times S^1$ that
  satisfies condition (1) for some $0<r_0<1$. Then $f$ is ambient isotopic rel
  $\partial \mathcal{M}$ to a diffeomorphism $f'$ of $\mathcal{M}$ such that
  $f'$ satisfies (1) for a possibly smaller $r_0>0$ and $f'_{|D^2 \times P_0}$
  is transverse to $D^2 \times P_0$.
\end{theorem}

\begin{proof}
  Choose and fix an embedding of $\mathcal{M}$ into $\mathbb{R}^3$. Let $d$ be the (Euclidean) distance between $f(A_{[1 - r_0,1 - \frac{r_0}{2}]} \times P_0)$ and $D^2 \times P_0$.
  Define $f_1$ as the restriction of $f$ to $A_{[0, 1 - r_0]} \times P_0$. There is a neighborhood $U$ of $f_1$ in $C_s^{\infty} (A_{[0, 1 - r_0]} \times P_0, \tmop{Int}
  \mathcal{M})$ consisting of smooth embeddings such that for any $g \in U, x \in A_{[0, 1 - r_0]} \times P_0$, $\norm{f_1(x)-g(x)} < \frac{d}{2}$ and $g(x) \not \in f(A_{[1 - \frac{r_0}{2}, 1]} \times P_0)$. Choose a smaller neighborhood $W$ of $f_1$ such that for any $h \in W, t \in [0,1]$, $(1-t)f_1 + th \in U$. We then take $h_1 \in W$ that is transverse to $\tmop{Int} D^2 \times P_0$, and define an isotopy $H$ by
  \[ H : A_{[0, 1 - r_0]} \times P_0 \times I \rightarrow \tmop{Int} \mathcal{M}, H (a,P_0,
     t) = (1-t)f(a,P_0) + th(a,P_0) \]

  Next we extend this isotopy to an ambient one. Let $\widehat{H}: A_{[0, 1 -
  r_0]} \times P_0 \times I \rightarrow \tmop{Int} \mathcal{M} \times I$ be the track of
  $H$. Apply Lemma \ref{VF} to the isotopy $H \circ (f^{- 1} \times \tmop{id}_I)$ of
  $f (A_{[0, 1 - r_0]} \times P_0)$ in $\tmop{Int} \mathcal{M}$, we get a
  compactly supported vector field $X$ on $\tmop{Int} \mathcal{M} \times I$
  that equals to the tangent vectors of curves $t \mapsto \widehat{H} (a, P_0, t)$ on
  $\tmop{Im} \widehat{H}$. Composing $X$ with the natural projection of tangent
  bundles from $T (\tmop{Int} \mathcal{M} \times I)$ onto $T (\tmop{Int}
  \mathcal{M})$, we obtain a compactly supported time dependent vector field
  $G$ on $\tmop{Int} \mathcal{M}$. By definition, for $(H (a, P_0, t), t) \in
  \widehat{H} (A_{[0, 1 - r_0]} \times P_0 \times I)$, the (Euclidean) norm of $G (H (a, P_0, t), t) =
  \frac{\partial}{\partial t} H (a, P_0, t) = h(a, P_0) - f(a, P_0)$ is less than $\frac{d}{2}$.
  
  Let $\rho : \tmop{Int} \mathcal{M} \times I \rightarrow I$ be a smooth
  function such that
  
  (i) $\rho$ is constantly 1 on $\tmop{Im} \widehat{H}$
  
  (ii) $\rho$ is constantly 0 on $f (A_{[1 - \frac{r_0}{2}, 1]} \times P_0) \times I$
  
  (iii) $\norm{ \rho (x, t) G (x, t) } < d$ for
  any $(x, t) \in \tmop{Int} \mathcal{M} \times I$
  
  Condition (i) and (ii) are possible since $f (A_{[1 - \frac{r_0}{2}, 1]} \times P_0)  \cap 
   \tmop{Im} H = \varnothing$ implies $(f (A_{[1 - \frac{r_0}{2}, 1]} \times P_0)
  \times I) \cap \tmop{Im} \widehat{H} = \varnothing$.
  
  The time dependent vector field $\rho G$ generates a compactly supported
  ambient isotopy $\Phi$ of $\tmop{Int} \mathcal{M}$. We can thus extend
  $\Phi$ by $\tmop{identity}$ to an ambient isotopy rel $\partial \mathcal{M}$
  of $\mathcal{M}$, which we still denote by $\Phi$. Since $\rho G = G$ on
  $\tmop{Im} \widehat{H}$, the isotopy $\Phi$ extends $H \circ (f^{- 1} \times
  \tmop{id}_I)_{| f(A_{[0, 1 - r_0]} \times P_0) \times I}$. Condition (ii) above ensures that $\Phi$ is stable on $f
  (A_{[1 - \frac{r_0}{2}, 1]} \times P_0)$. Finally, $\Phi_1$ moves any point by a
  distance less than $d$, therefore $\Phi_1 \circ f (A_{[1 - r_0, 1 - \frac{r_0}{2}]}
  \times P_0) \cap (D^2 \times P_0) = \varnothing$.
  
  Set $f'=\Phi_1 \circ f$. The transversality of $f'_{|D^2 \times P_0}$ to $D^2 \times
  P_0$ follows from the observation that $\Phi_1 \circ f = H_1$ on $A_{[0, 1 - r_0]} \times P_0$ while $\Phi_1 \circ f
  = f$ on $A_{[1 - \frac{r_0}{2}, 1]} \times P_0$.
\end{proof}

\subsection{The Third Step}

We are now ready to proceed with the third step.

\begin{theorem} \label{3rd}
  Let $f$ be a diffeomorphism of $\mathcal{M}= D^2 \times S^1$ such that $f$
  satisfies (1) for some $r_0 > 0$ and $f_{|D^2 \times P_0}$ is transverse to
  $D^2 \times P_0$., then $f$ is ambient isotopic rel $\partial \mathcal{M}$
  to a diffeomorphism $f'$ such that $f'$ satisfies (1) for possibly smaller
  $r_0$ and $f' (D^2 \times P_0) \cap (D^2 \times P_0) = \partial D^2 \times
  P_0$.
\end{theorem}

\begin{proof}
  If $f (D^2 \times P_0) \cap (D^2 \times P_0) = \partial D^2 \times P_0$,
  there is then nothing to prove. So we may assume $f (D^2 \times P_0) \cap
  (D^2 \times P_0)$ is a finite union of more than one circles. Choose a
  circle $S$ in this family that is not inside any other with respect to $f
  (D^2 \times P_0)$. Then $S$ bounds a disk $D_0' = f (D_0)$ in $f (D^2 \times
  P_0)$, where $D_0$ is a disk in $D^2 \times P_0$. On the other hand, $S$
  bounds a disk $D_0''$ in $D^2 \times P_0$. Denote the union $D_0' \cap D_0''$ by $S_0$. Then $S_0 \cap (D^2 \times P_0)=D_0''$. Intuitively, $S_0$ is sphere touching
  exactly one side of $D^2 \times P_0$. To be precise, for an interval $J$,
  define $T_J = \{ (x, e^{t i}) \in D^2 \times S^1 |t \in J \}$. Then there
  exist $a \in (0, 2 \pi)$ such that either $S_0 \cap T_{[- a, 0)} =
  \varnothing$ or $S_0 \cap T_{(0, a]} = \varnothing$. Without lost of
  generality, we assume $S_0 \cap T_{[- a, 0)} = \varnothing$ (the property
  (1) plays no role in the proof of this theorem, so the two cases are
  symmetric).
  
  We wish to push $D_0'$ to the ``negative side'' of $D^2 \times P_0$ without
  affecting other circles in $f (D^2 \times P_0) \cap (D^2 \times P_0)$. This
  will reduce the number of circles by one.
  
  The first step to achieve this is to position $S_0$ in the interior of a
  suitably chosen $C^1$ 3-ball. By our transversality assumption, we may
  choose a closed collar $C$ of $\partial D_0$ in $D^2 \times P_0 - \tmop{Int}
  D_0$ with the following properties:
  
  (i) $f (C) \cap (D^2 \times P_0) = S$
  
  (ii) $f (C) \subset T_{(- a / 2, 0]}$
  
  (iii) The subspace $K \assign \overline{D^2 \times P_0 - (C \cup D_0)}$ is
  homeomorphic to an annulus.
  
  Condition (iii) is made possible by the Annulus Theorem. Note that $f (K)
  \cap D_0' = \varnothing$ and $f (K) \cap D_0'' = \varnothing$ by our choice.
  
  The disk $D_0'$ is a neat submanifold of $\tmop{Int} \mathcal{M}- T_{[- a,
  0)}$. Hence we can take a closed tubular neighborhood $Q:D_0' \times [-1,1] \rightarrow \tmop{Int} \mathcal{M}- T_{[- a,
  0)}$ of $D_0'$ in
  $\tmop{Int} \mathcal{M}- T_{[- a, 0)}$ such that $Q_{|\partial D_0' \times [-1,1]}$ is a (closed) tubular neighborhood of $\partial D_0'$ in $\tmop{Int}D^2 \times P_0$ (c.f.
  \cite{Hirsch} Theorem 4.6.4.). We may also assume that $Q$ is disjoint from $f
  (K)$. Without loss of generality, we may assume $Q(D_0' \times \{ 1 \}) \cap D_0'' = \varnothing$ while $Q(D_0' \times \{ -1 \}) \cap D_0'' \neq \varnothing$. Define $D_1':=Q(D_0' \times {1})$. It is not hard to
  observe that $D_1' \cap (D^2 \times P_0) = \partial D_1'$, and that $D_1'$ is transverse to $D^2 \times P_0$. Define $D_1'':=D_0'' \cup Q(\partial D_0' \times [0,1])$. Then $D_1''$ is the disk bounded inside $\partial D_1'$ with respect to $D_2 \times P_0$. Note that $(D_1' \cup D_1'') \cap f(K) \subseteq (Q \cup D_0'') \cap f(K) = \varnothing$.
  
  Now we extend $D_1'$ to a $C^1$ sphere.
  
  {\tmstrong{Claim}}: There exist a $C^1$ embedding $i_1 : S^2 \rightarrow
  \tmop{Int} \mathcal{M}$ such that
  
  (i) $i_1 (D_+) = D_1'$
  
  (ii) $i_1 (\tmop{Int} D_-) \subset \tmop{Int} \mathcal{M} \cap T_{(- a / 4,
  0)}$
  
  (iii) $\tmop{Im} i_1$ bounds a topological 3-ball $\widetilde{B}_1$ that is disjoint from
  $f (K)$
  
  (iv) $S_0 \subset \tmop{Int} \widetilde{B}_1$
  
  where $D_+, D_-$ are the upper and lower hemisphere of $S^2$ respectively.
  
  {\tmstrong{Proof of the claim}}: The above claim is intuitively obvious, yet
  its proof is rather long. So we postpone that proof and make it a separate
  lemma below (Lemma \ref{Lemmaforproof}).
  
  Note that condition (i) and (ii) above implies $\tmop{Im} i_1 \subset
  \tmop{Int} \mathcal{M}- T_{[- a, - a / 4]}$.
  
  Having the desired $i_1$, we shall approximate it with a $C^{\infty}$
  embedding $i_2 : S^2 \rightarrow \tmop{Int} \mathcal{M}$. In the space $C^1
  (S^2, \tmop{Int} \mathcal{M})$ equipped with Whitney topology, the set of
  embeddings is open and the set of $C^{\infty}$ maps is dense. Hence we can
  find a $C^{\infty}$ embedding $i_2$ that is arbitrarily close to $i_1$. It
  should be close enough so that
  
  (i) The $C^\infty$ ball $\widetilde{B}_2$ bounded by $\tmop{Im} i_2$ is disjoint with
  $f (K) \cup (D^2 \times \{ e^{- a i / 2} \})$
  
  (ii) $\tmop{Im} i_2 \cap (D^2 \times P_0)$ is a smooth circle $S'$
  
  (iii) $S_0 \subset \tmop{Int} \widetilde{B}_2$
  
  Condition (i) is easy to guarantee. In order to achieve (ii), we identify
  $T_{(-\frac{a}{2},\frac{a}{2})}$ naturally with
  $\tmop{Int} D^2 \times (-\frac{a}{2},\frac{a}{2})$ by identifying $(z, e^{t
  i})$ with $(z, t)$. Let $p : \tmop{Int} D^2 \times (-\frac{a}{2},\frac{a}{2}) \rightarrow (-\frac{a}{2},\frac{a}{2})$ be the projection onto the last coordinate.  Consider the map $N : S^1 \times (- 1, 1)
  \rightarrow S^2$ defined by $N (z, t) = \left( \sqrt{1 - t^2} \frac{z}{\norm{ z
  }}, t \right)$. Choose $\delta > 0$ small enough so that $i_1 \circ N (S^1 \times [- \delta, \delta]) \subset T_{(-\frac{a}{2},\frac{a}{2})}$ and that
  $\frac{\partial}{\partial t} p \circ i_1 \circ N (z, t)$ is positive for $(z, t)
  \in S^1 \times [- \delta, \delta]$. We can require $i_2$ to map $N (S^1
  \times \{ \delta \})$ (resp. $N (S^1 \times \{ - \delta \})$) into
  $\tmop{Int} D^2 \times (0, a / 2)$ (resp. $\tmop{Int} D^2 \times (-
  a / 2, 0)$), and require that $\frac{\partial}{\partial t} p \circ i_2 \circ N (z,
  t) > 0$ on $S^1 \times [- \delta, \delta]$. By
  monotonicity, for each $z \in S^1$, $i_2 \circ N (\{ z \} \times [- \delta,
  \delta]) \cap (\tmop{Int} D^2 \times \{ 0 \})$ is a singleton, and the
  correspondence $z \mapsto i_2 \circ N (\{ z \} \times [- \delta, \delta])
  \cap (\tmop{Int} D^2 \times \{ 0 \})$ is a smooth embedding by the Implicit
  Function Theorem. If we require further that $i_2 (S^2 - N (S^1 \times (-
  \delta, \delta))) \cap (\tmop{Int} D^2 \times \{ 0 \}) = \varnothing$, we
  ensure (ii). Finally, we may require that $i_1$ and $i_2$ are close enough such that with respect to some embedding of $\mathcal{M}$ in $\mathbb{R}^3$, the linear homotopy from $i_1$ to $i_2$ takes place in $\tmop{Int} \mathcal{M} - S_0$. Hence $S_0 \subset \tmop{Int} \widetilde{B}_2$.
  
  As a consequence of the conditions (i) and (iii) of $i_2$, the circle $S'$
  is disjoint from both $f (K)$ and $D_0' = f (D_0)$. In fact, it is also
  disjoint from $f (C)$, for $f (C) \cap (D^2 \times P_0) = \partial D_0'
  \subset S_0$ by our choice. Thus $S'$ is disjoint from the entire $f (D^2
  \times P_0)$.
  
  Now we are prepared to apply Lemma \ref{Compress}. Identify $\widetilde{B}_2$ with
  $B^3$. Let $D$ be the disk bounded by $S'$ in $D^2 \times P_0$. Set
  $\widetilde{D} = \partial \widetilde{B}_2 \cap T_{(- a / 2, 0]}$. Let $U$ be a
  tubular neighborhood of $S'$ in $\partial \widetilde{B}_2$ such that $U \cap
  f (D^2 \times P_0) = \varnothing$. Recall that we require $f (C) \subset
  T_{(- a / 2, 0]}$, so $f (C) \cap \partial \widetilde{B}_2 \subset \widetilde{D}
  - U$. According the Lemma \ref{Compress}, there exist a smooth isotopy $H :
  \widetilde{B}_2 \times I \rightarrow \widetilde{B}_2$ with $H_0 = \tmop{id},
  H_1 (\widetilde{B}_2) \cap (D^2 \times P_0) = \varnothing$ and $H$ is stable
  on a neighborhood $W$ of $\widetilde{D} - U$ in $\widetilde{B}_2$.
  
  Next we shall extend this isotopy to a compactly supported ambient isotopy
  on a neighborhood of $\widetilde{B}_2$. Choose a neighborhood $O$ of
  $\widetilde{B}_2$ in $\tmop{Int} \mathcal{M}$ such that $O$ is disjoint from
  a neighborhood of $f (K)$. By Lemma \ref{VF}, there is a compactly supported
  vector field $X$ on $O \times I$ that equals to the tangent vectors of $t
  \mapsto \widehat{H} (x, t)$ on $\widehat{H} (\widetilde{B}_2 \times I)$, where $\widehat{H}$ stands for the track of $H$. Composing
  $X$ with the natural projection $T (O \times I) \rightarrow T O$, we obtain
  a compactly supported time dependent vector field $G$. The stability of $H$
  on $W$ implies $X \equiv 0$ on $W \times I$, hence $G \equiv 0$ on $W \times
  I$.
  
  We then look for a smooth function $\xi : O \times I \rightarrow I$ such
  that:
  
  (i) $\xi G = G$ on $\widehat{H} (\widetilde{B}_2 \times I)$
  
  (ii) $\xi G = 0$ on $E \times I$, where $E = f (C) \cap O - \tmop{Int}
  \widetilde{B}_2$
  
  Since $f (C) \cap \partial \widetilde{B}_2 \subset \widetilde{D} - U$, we know
  that $E \cap \widetilde{B}_2 \subset \widetilde{D} - U \subset W$. Thus we can
  take a neighborhood $V$ of $E$ in $O$ with $V \cap \widetilde{B}_2 \subset
  W$. The set $f (C) \cap \tmop{supp}G - \tmop{Int} \widetilde{B}_2$ is
  compact and lies in $V$. Let $\xi : O \times I \rightarrow I$ be a smooth
  function with $\xi = 0$ on $(f (C) \cap \tmop{supp}G - \tmop{Int}
  \widetilde{B}_2) \times I$ and $\xi = 1$ outside $V \times I$. Condition
  (ii) is easy to check. Condition (i) is certainly true outside $V \times I$,
  while $\widehat{H} (\widetilde{B}_2 \times I) \cap (V \times I) \subset
  (\widetilde{B}_2 \times I) \cap (V \times I) \subset W \times I$. Since $G$
  vanishes on $W \times I$, condition (i) is valid on all of $\widehat{H}
  (\widetilde{B}_2 \times I)$.
  
  Now $\xi G$ generates a compactly supported ambient isotopy of $O$, which
  can be extended to one of $\mathcal{M}$ by identity. Denote this isotopy by
  $\widetilde{H}$. Then $\widetilde{H}$ is stable in a neighborhood of $\partial
  \mathcal{M}$, hence $f' \assign \widetilde{H}_1 \circ f$ satisfy the
  property (1) for a possibly smaller $r_0 > 0$. Since $\widetilde{H}$ extends
  $H$, we have $\widetilde{H}_1 (\widetilde{B}_2) \cap (D^2 \times P_0) =
  \varnothing$. The condition (ii) of $\xi G$ implies $\widetilde{H}$ is stable on
  $f (C) - \widetilde{B}_2$. Our requirement of $O$ guarantees the stability
  of $\widetilde{H}$ on a neighborhood $f (K)$. Thus $\widetilde{H}$ is stable on $f
  (K) \cup (f (C) - \widetilde{B}_2)$. Since $f (K) \cap \widetilde{B}_2 =
  \varnothing$, we have $f (K) \cup (f (C) - \widetilde{B}_2) = f (K) \cup f
  (C) - \widetilde{B}_2$.
  
  We can check that $f'$ reduces the number of circles in the intersection by
  one, since we have
  \begin{align*}
    & f' (D^2 \times P_0) \cap (D^2 \times P_0) & &\\
    = & \widetilde{H}_1 \circ f (D^2 \times P_0) \cap (D^2 \times P_0)& &\\
    = & \widetilde{H}_1 (f (D^2 \times P_0) - \widetilde{B}_2) \cap (D^2
    \times P_0) & &\tmop{for} \widetilde{H}_1 (\widetilde{B}_2) \cap (D^2
    \times P_0) = \varnothing \\
    = & \widetilde{H}_1 (f (K) \cup f (C) - \widetilde{B}_2) \cap (D^2 \times
    P_0) & & \tmop{for} f (D_0) = D_0' \subset B_0 \subset \tmop{Int}
    \widetilde{B}_2\\
    = & (f (K) \cup f (C) - \widetilde{B}_2) \cap (D^2 \times P_0) & & \tmop{by}
    \tmop{stability} \tmop{of} \widetilde{H} \tmop{on} f (K) \cup f (C) -
    \widetilde{B}_2 \\
    = & (f (K) \cup (f (C) - \widetilde{B}_2) ) \cap (D^2 \times P_0) & &\\
    = & f (K) \cap (D^2 \times P_0) & & \tmop{for} f (C) \cap (D^2 \times P_0) =
    S \subset \widetilde{B}_2 \\
    = & f (D^2 \times P_0) \cap (D^2 \times P_0) - S & & \tmop{for} f (C \cup D_0) \cap (D^2 \times P_0) = S  \\    
  \end{align*}
  Finally, since $f' (D^2 \times P_0) \cap (D^2 \times P_0) = f (D^2 \times
  P_0) \cap (D^2 \times P_0) - S \subset f (K)$, the stability of $\widetilde{H}$
  near $f (K)$ ensures transversality of $f'_{|D^2 \times P_0}$ with $D^2
  \times P_0$. The theorem now follows by induction.
\end{proof}

\begin{lemma} \label{Lemmaforproof}
  The claim in the proof of Theorem \ref{3rd} is valid.
\end{lemma}

\begin{proof}
  We use the same notations as in the proof of Theorem \ref{3rd}. Define $i_{1 | D_+}$ by choosing a diffeomorphism between $D_+$ and $D_1'$. We will extend it to a $C^1$ embedding $i_1 : S^2 \rightarrow \tmop{Int} \mathcal{M}$ such that $i_1 (\tmop{Int} D_-) \subset \tmop{Int} \mathcal{M} \cap T_{(- a / 4, 0)}$. This would satisfy condition (i) and (ii). The image of the extended $i_1$ is a $C^1$ 2-sphere, thus has a $C^1$ tubular neighborhood (cf. \cite{Hirsch} Chapter 4, Exercise 1). Hence by the Schoenflies Theorem, it bounds a topological 3-ball $\widetilde{B_1}$. If $i_1 (D_-)$ is sufficiently close to the disk $D_1''$, the image of $i_1$ would be disjoint from $f(K)$. This would guarantee condition (iii), for $f(K)$ is path-connected and contains $\partial D^2 \times P_0$. We now carry out this plan.
  
  The correspondence
  $(z, e^{t i}) \mapsto (z, t)$ identifies $T_{[-a / 4, a / 4]}$ with the subspace $\tmop{Int} D^2 \times (- a / 4, a / 4)$ of $\mathbb{R}^3$.
  
  Let $N : S^1 \times (- 1, 1) \rightarrow S^2$ be defined by $N (z, t) =
  \left( \sqrt{1 - t^2} \frac{z}{\norm{z}}, t \right)$. Given $z \in S^1$, the
  tangent vector of the curve $t \mapsto i_1 \circ N (z, t)$ at $t = 0$ is of
  the form $(X_1 (z), X_2 (z), X_3 (z))$. By our choice of $i_1$, $X_3 (z) >
  0$ for all $z \in S^1$. Applying Lemma \ref{rescale} with $X = X_3$, we can reduce the
  problem to the case where $X_3 (z) \equiv 1$.
  
  Define $h_i : S^1 \rightarrow \mathbb{R}, i = 1, 2$ by $i_1 (z) = (h_1 (z),
  h_2 (z), 0)$ for any $z \in S^1 \subset D_+$.
  
  Choose $\varepsilon > 0$ with $4 \varepsilon < a/4$, choose a smooth function $\rho : [- \varepsilon, 0]
  \rightarrow [- \varepsilon, 0]$ with the properties:
  
  (i) $\rho$ is non-decreasing
  
  (ii) $\rho' (0) = 1, \rho (0) = 0$
  
  (iii) $\rho' (- \varepsilon) = 0$
  
  Define $H : S^1 \times [- \varepsilon, 0] \rightarrow \mathbb{R}^3$ by
  \[ H (z, t) = (h_1 (z) + X_1 (z) \rho (t), h_2 (z) + X_2 (z) \rho (t), t) \]
  Then $H_{|S^1 \times \{ 0 \}} = i_{1| S^1}$. For $\varepsilon$ small enough
  we can ensure $\tmop{Im} H \subset \tmop{Int} D^2 \times [- \varepsilon,
  0]$. Since $S^1$ is compact and the collection of embeddings is open in $C^1
  (S^1, \mathbb{R}^2)$, we may further shrink $\varepsilon$ to make $H$ a
  smooth embedding.
  
  For any $z \in S^1$, $\frac{\partial}{\partial t} H (z, t) = (X_1 (z) \rho'
  (t), X_2 (z) \rho' (t), 1)$, so $\frac{\partial}{\partial t}_{|t = 0} H (z,
  t) = (X_1 (z), X_2 (z), 1)$ and $\frac{\partial}{\partial t}_{|t = -
  \varepsilon} H (z, t) = (0, 0, 1)$. This implies $i_1$ can be extended
  by $H$ to a $C^1$ embedding of $D_+ \cup N (S^1 \times [- \varepsilon, 0])$
  into $\tmop{Int} \mathcal{M}$.
  
  The region bounded by $H (S^1 \times \{ - \varepsilon \})$ in $\mathbb{R}^2
  \times \{ - \varepsilon \}$ is a smooth disk $D_2 \times \{ - \varepsilon
  \}$. Since a self-diffeomorphism of $S^1$ extends to one of $D^2$, we may
  choose a diffeomorphism $\varphi : D^2 \rightarrow D_2$ with $\varphi_{|
  \partial D^2} = H_{- \varepsilon}$.
  
  Now choose a smooth function $\lambda : [- 4 \varepsilon, - \varepsilon]
  \rightarrow \mathbb{R}$ such that
  
  (i) $\lambda (t) = t$ if $t \in [- 2 \varepsilon, - \varepsilon]$
  
  (ii) $\lambda (t) \equiv - 4 \varepsilon$ if $t \in [- 4 \varepsilon, - 3
  \varepsilon]$
  
  (iii) $\lambda (t)$ is strictly increasing on $[- 3 \varepsilon, -
  \varepsilon]$
  
  and a smooth function $\tau : [- 4 \varepsilon, - \varepsilon] \rightarrow
  \mathbb{R}$ such that
  
  (i) $\tau (t) = - \varepsilon$ if $t \in [- 2 \varepsilon, - \varepsilon]$
  
  (ii) $\tau (t) = t$ if $t \in [- 4 \varepsilon, - 3 \varepsilon]$
  
  (iii) $\tau$ is non-decreasing
  
  We then extend $H$ to $S^1 \times [- 4 \varepsilon, 0]$ by defining
  \[ H (z, t) = (\varphi ((1 + \varepsilon + \tau (t)) z), \lambda (t)), z \in
     S^1, t \in [- 4 \varepsilon, - \varepsilon] \]
  In particular, $H (z, t) = (\varphi (z), t)$ on $S^1 \times [- 2
  \varepsilon, - \varepsilon]$. Hence we may extend $i_1$ by $H$ to a $C^1$
  map from $D_+ \cup N (S^1 \times [- 4 \varepsilon, 0])$ to $\tmop{Int} \mathcal{M}$.
  Our choice of $\varphi, \lambda, \tau$ ensures this map is an embedding.
  
  On the other hand, $H (z, t) = (\varphi ((1 + \varepsilon + t) z), - 4
  \varepsilon)$ on $S^1 \times [- 4 \varepsilon, - 3 \varepsilon]$. Let $p :
  D_- \rightarrow D^2$ be the vertical projection onto the plane
  $\mathbb{R}^2$. Then $p \circ N (S^1 \times [- 4 \varepsilon, - 3
  \varepsilon])$ is the annulus $A = \left\{ x \in D^2 | \sqrt{1 - 16
  \varepsilon^2} \leqslant \norm{x} \leqslant \sqrt{1 - 9 \varepsilon^2}
  \right\}$. For $x \in A$, we compute
  \[ i_1 \circ p^{- 1} (x) = H \circ N^{- 1} \circ p^{- 1} (x) = \left(
     \varphi \left( \left( 1 + \varepsilon - \sqrt{1 - {\norm{x}}^2} \right)
     \frac{x}{\norm{x}} \right), - 4 \varepsilon \right) \]
 
  Writing $x = t z$ for $t \in \left[ \sqrt{1 - 16 \varepsilon^2}, \sqrt{1 - 9
  \varepsilon^2} \right], z \in S^1$, we have
  \[ i_1 \circ p^{- 1} (t z) = \left( \varphi \left( \left( 1 + \varepsilon -
     \sqrt{1 - t^2} \right) z \right), - 4 \varepsilon \right) \]
  
  Choose a smooth function $\sigma : \left[ 0, \sqrt{1 - 9 \varepsilon^2}
  \right] \rightarrow \mathbb{R}$ satisfying
  
  (i) $\sigma$ is strictly increasing
  
  (ii) $\sigma (t) = t$ in a neighborhood of $0$
  
  (iii) $\sigma (t) = 1 + \varepsilon - \sqrt{1 - t^2}$ in a neighborhood of
  $\left[ \sqrt{1 - 16 \varepsilon^2}, \sqrt{1 - 9 \varepsilon^2} \right]$
 
 Define a diffeomorphism $\phi$ from the closed disk $\bar{B}_0 \left(
  \sqrt{1 - 9 \varepsilon^2} \right) \subset \mathbb{R}^2$ to $\bar{B}_0 (1 -
  2 \varepsilon)$ by $\phi (t z) = \sigma (t) z, z \in S^1, t \in \left[ 0,
  \sqrt{1 - 9 \varepsilon^2} \right]$. We could now extend $i_1$ to $S^2$ by
  \[ i_1 \circ p^{- 1} (x) = (\varphi \circ \phi (x), - 4 \varepsilon), p^{-
     1} (x) \in p^{- 1} \left( \bar{B}_0 \left( \sqrt{1 - 9 \varepsilon^2}
     \right) \right) \]
  The extended $i_1$ is a $C^1$ embedding, and for small enough $\varepsilon$
  the image $i_1 (D_-)$ will be close enough to $D_1''$.
  
  It remains to verify condition (iv). Condition (i) and (ii) implies that $S_0 \cap \tmop{Im} i_1 = \varnothing$. Suppose $S_0$ is outside $\widetilde{B}_1$, then it is easy to verify that $D_2 \times P_0 - \partial D_1'$ is outside $\widetilde{B}_1$. But $D_2 \times P_0$ is transverse to $D_1'$, hence one can choose a curve in $D_2 \times P_0$ intersecting $\tmop{Im} i_1$ transversely and lies outside $\widetilde{B}_1$ except for one point. Take a ($C^1$) tubular neighborhood of $\tmop{Im} i_1$ in $\tmop{Int} \mathcal{M}$. As a consequence of the above discussion, this tubular neighborhood would be disjoint from the interior of $\widetilde{B}_1$, which is impossible. Thus we have $S_0 \subset \tmop{Int} \widetilde{B}_1$.
\end{proof}
 
\subsection{The Fourth Step}

Now we can perform the fourth step.

\begin{theorem} \label{4th}
  Let $f$ be a diffeomorphism of $\mathcal{M}= D^2 \times S^1$ such that $f$
  satisfies (1) for some $r_0 > 0$ and $f (D^2 \times P_0) \cap (D^2 \times
  P_0) = \partial D^2 \times P_0$. Then $f$ is ambient isotopic $\partial
  \mathcal{M}$ to a diffeomorphism $f'$ such that $f' (D^2 \times P_0) = D^2
  \times P_0$ and $f' = \tmop{id}$ on a neighborhood of $\partial
  \mathcal{M}$.
\end{theorem}

\begin{proof}
  For any interval $J$, define $T_J = \{ (x, e^{t i}) \in D^2 \times S^1 |t
  \in J \}$, $A_J = \{ x \in D^2 | \norm{x} \in J \}$ and $E_J = \partial B^3
  \cap (\mathbb{R}^2 \times J)$.
  
  Shrinking $r_0$ further, we may assume that $3 r_0<1$ and $f$ satisfy (1) for $0 \leqslant
  r \leqslant 3 r_0$, and that
  \begin{align*}
   (A_{[1 - 3 r_0, 1]} \times S^1) \cap f (D^2 \times P_0) & = f (A_{[1 - 3
     r_0, 1]} \times P_0) \\
     &  = \{ (z, e^{(1 - \norm{z}) i}) \in D^2 \times S^1 |1 -
     3 r_0 \leqslant \norm{z} \leqslant 1 \} 
  \end{align*}

  We will bend $f_{|A_{[1 - 2 r_0, 1]} \times S^1}$ toward identity by an
  ambient isotopy.
  
  Let $\rho : I \rightarrow [- 1, 0]$ be a smooth function such that
  
  (i) $ r-1 < \rho (r) < 0$ for $r \in (1 - 3 r_0, 1 - 2 r_0)$
  
  (ii) $\rho \equiv 0$ on $[0, 1 - 3 r_0]$
  
  (iii) $\rho (r) = r-1$ for $r \in [1 - 2 r_0, 1]$
  
  Define $\widetilde{H} : D^2 \times S^1 \times I \rightarrow D^2 \times S^1$ by
  $\widetilde{H} (x, z, t) = (x, e^{t \rho (\norm{x}) i} z)$. It is not hard to check
  that $\widetilde{H}$ is an ambient isotopy of $\mathcal{M}$ and $\widetilde{f}
  \assign \widetilde{H_1} \circ f$ satisfy the following properties
  
  (i) $\widetilde{f} = \tmop{id}$ on $A_{[1 - 2 r_0, 1]} \times S^1$
  
  (ii) $\widetilde{f} (D^2 \times P_0) \cap (D^2 \times P_0) = A_{[1 - 2 r_0, 1]}
  \times P_0$
  
  (iii) $\widetilde{f} (D^2 \times P_0) \cap T_{(- \varepsilon, 0)} = \varnothing$
  for some $\varepsilon > 0$
  
  Next, we take a smooth embedding $G : B^3 \rightarrow T_{(- \varepsilon, 0]} \cap
  \tmop{Int} \mathcal{M}$ such that
  
  (i) $G (E_{[- 2 / 3, 1]}) = A_{[0, 1 - r_0]} \times P_0 = G (B^3) \cap (D^2
  \times P_0)$
  
  (ii) For any $(w, z) \in E_{[- 2 / 3, 0]}, w \in \mathbb{R}^2, z \in [- 2 /
  3, 0]$, 
  \begin{equation*}
  G (w, z) = \left( \left( 1 - \left( \frac{3 z}{2} + 2 \right) r_0
  \right) \frac{w}{\norm{w}}, P_0 \right)
  \end{equation*}

  Note that condition (iii) of $\widetilde{f}$ implies $G (B^3) \cap \widetilde{f} (D^2 \times P_0) = A_{[1 - 2 r_0, 1 -
  r_0]} \times P_0$. Also, condition (ii) of $G$ implies $G (E_{[- 2 / 3, 0]})
  = A_{[1 - 2 r_0, 1 - r_0]} \times P_0$.
  
  We could now choose a smooth embedding $h_0 : S^2 \hookrightarrow \tmop{Int}
  \mathcal{M}$ with the properties:
  
  (i) $h_0 (x) = G (x)$ for $x \in E_{[- 1, 0]}$
  
  (ii) $h_0 (E_{[- 2 / 3, 1]}) = \widetilde{f} (A_{[0, 1 - r_0]} \times P_0)$
  
  Define $S = h_0 (S^2)$. Denote by $B$ the 3-ball bounded within $S$. It is
  not hard to verify that $(A_{[1 - r_0, 1]} \times P_0) \cap B = (A_{[1 -
  r_0, 1]} \times P_0) \cap S = A_{[1 - r_0, 1 - r_0]} \times P_0 = G (E_{[- 2
  / 3, \dotminus 2 / 3]})$. By \cite{MunkresIsotopy2sphere}, $h_0 : S^2 \rightarrow S$ extends to a diffeomorphism $H_0 : B^3
  \rightarrow B$.
  
  The ball $B' := G (B^3)$ sits inside the ball $B$. We shall construct an
  isotopy $\Psi$ of $B$ so that $\Psi_1 (B) = B'$.
  
  Define $C_0 : \partial B^3 \times [0, 1) \rightarrow B^3$ by $C_0 (x, t) =
  (1 - t) x$. Let $C = H_0^{- 1} \circ G \circ C_0 : E_{[- 1, 0]} \times [0,
  1) \rightarrow B^3$. Then for any $x \in E_{[- 1, 0]}$, $C (x, 0) = H_0^{- 1} \circ G \circ C_0 (x, 0) =
  H_0^{- 1} \circ G (x) = H^{- 1}_0 \circ h_0 (x) = x$ and
  $C (E_{[- 1, 0]} \times (0, 1)) \subset \tmop{Int} B^3$. Applying Lemma \ref{partialextension},
  we obtain a collar $C' : \partial B^3 \times [0, \delta] \rightarrow B^3$
  for some $\delta > 0$ such that $C' = C$ on $E_{[- 1, - 1 / 3]} \times [0,
  \delta]$. Applying Theorem \ref{IUC} and shrink $\delta$ if necessary, we can find
  a diffeomorphism $\varphi : B^3 \rightarrow B^3$ such that $\varphi =
  \tmop{id}$ on $\partial B^3$ and $\varphi \circ C' = C_0$ on $\partial B^3
  \times [0, \delta]$.
  
  Define $H \assign H_0 \circ \varphi^{- 1} : B^3 \rightarrow B$. On $E_{[- 1,
  - 1 / 3]} \times [0, \delta]$ we have
  \begin{align*}
   & H \circ C_0 & \\
    = & H_0 \circ \varphi^{- 1} \circ \varphi \circ C' & \\
    = & H_0 \circ C & \\
    = & G \circ C_0 & 
  \end{align*}
  
  Take an isotopy $F : B^3 \times I \rightarrow B^3$ such that $F_0 =
  \tmop{id}$, $F_1 (B^3) \subset C_0 (E_{[- 1, - 1 / 3]} \times [0, \delta])$
  and $F$ is stable on a neighborhood of $E_{[- 1, - 2 / 3]}$. Define
  isotopies $\Phi = H \circ F \circ (H^{- 1} \times \tmop{id}_I) : B \times I
  \rightarrow B$ and $\Phi' : G \circ F \circ (G^{- 1} \times \tmop{id}_I) :
  B' \times I \rightarrow B'$. The isotopy $\Phi$ compresses the ball $B$ to
  $\Phi_1 (B) = H \circ F_1 \circ H^{- 1} (B) = H \circ F_1 (B^3)$ while
  $\Phi'$ compresses $B'$ to $\Phi'_1 (B') = G \circ F_1 \circ G^{- 1} (B') =
  G \circ F_1 (B^3)$. Since $F_1 (B^3) \subset C_0 (E_{[- 1, - 1 / 3]} \times
  [0, \delta])$ and $H \circ C_0 = G \circ C_0$ on $E_{[- 1, - 1 / 3]} \times
  [0, \delta]$, we have $H \circ F_1 = G \circ F_1$, hence $\Phi'_1 (B') =
  \Phi_1 (B)$. Stability of $F$ implies that $\Phi$ is stable on a
  neighborhood of $H (E_{[- 1, - 2 / 3]})$ in $B$ and $\Phi'$ is stable on a
  neighborhood of $G (E_{[- 1, - 2 / 3]})$ in $B'$. Note that $G (E_{[- 1, - 2
  / 3]}) = H (E_{[- 1, - 2 / 3]})$.
  
  The two isotopies $\Phi$ and $\Phi' \circ (\Phi_1'^{-1} \times \tmop{id}_I) \circ
  (\Phi_1 \times \tmop{id}_I)$ of $B$ coincide at $t = 1$. Join them at $t =
  1$ and smooth the joint as usual, we obtain an isotopy $\Psi : B \times I
  \rightarrow B$ with properties:
  
  (i) $\Psi_0 = \tmop{id}$
  
  (ii) $\Psi_1 = \Phi_1^{\prime - 1} \circ \Phi_1$
  
  (iii) $\Psi$ is stable on a neighborhood $W$ of $G (E_{[- 1, - 2 / 3]})$
  
  The property (ii) of $\Psi$ implies that $\Psi_1 (B) = B'$.
  
  We compute
  \begin{align*}
   & \Psi_1 (\widetilde{f} (A_{[0, 1 - r_0]}) \times P_0) & &\\
    = & \Psi_1 (h_0 (E_{[- 2 / 3, 1]})) & & \tmop{by} \tmop{property}
    (\tmop{ii}) \tmop{of} h_0\\
    = & \Psi_1 (H (E_{[- 2 / 3, 1]})) & & \tmop{for} h_0 = H_{0| S^2} \tmop{and}
    \varphi_{|s^2} = \tmop{id}\\
    = & \Phi_1^{\prime - 1} \circ \Phi_1 \circ H (E_{[- 2 / 3, 1]}) & & \\
    = & G \circ F_1^{- 1} \circ G^{- 1} \circ H \circ F_1 (E_{[- 2 / 3, 1]}) & &
    \\
    = & G \circ F_1^{- 1} \tmmathbf{} \circ G^{- 1} \circ G \circ F_1 (E_{[- 2
    / 3, 1]}) & & \tmop{for} G \circ F_1 = H \circ F_1\\
    = & G (E_{[- 2 / 3, 1]}) &  \\
    = & A_{[0, 1 - r_0]} \times P_0 & & \tmop{by} \tmop{property (i)} \tmop{of}
    G
  \end{align*}

  Finally, we wish to extend $\Psi$ to $\mathcal{M}$. Apply Lemma \ref{VF} to
  $\Psi : B \times I \rightarrow \tmop{Int} \mathcal{M}$, we obtain a
  compactly supported vector field $X$ on $\tmop{Int} \mathcal{M} \times I$
  that equals to the tangent vectors $\frac{\partial}{\partial t} \Psi (x,
  t)$. Extend $X$ to $\mathcal{M} \times I$ by $0$ and compose the result with
  the projection $T (\mathcal{M} \times I) \rightarrow T\mathcal{M}$, we get a
  time-dependent vector field $G$. Since $(A_{[1 - r_0, 1]} \times P_0) \cap B
  = G (E_{[- 2 / 3, \dotminus 2 / 3]}) \subset G (E_{[- 1, \dotminus 2 / 3]})
  \subset W$, we can take a neighborhood $V$ of $(A_{[1 - r_0, 1]} \times P_0)
  $ in $\mathcal{M}$ with $V \cap B \subset W$. Let $\xi : \mathcal{M}
  \times I \rightarrow I$ be a smooth function that equals to $0$ on $A_{[1 -
  r_0, 1]} \times P_0 \times I$ and equals to $1$ outside $V \times I$. Then
  $\xi G$ generates a compactly supported ambient isotopy $\widetilde{\Psi}$ of $\mathcal{M}$ such
  that
  
  (i) $\widetilde{\Psi}$ extends $\Psi$
  
  (ii) $\widetilde{\Psi}$ is stable on $A_{[1 - r_0, 1]} \times P_0$
  
  (iii) $\widetilde{\Psi}$ is stable on a neighborhood of $\partial \mathcal{M}$
 
 Now $f' = \widetilde{\Psi}_1 \circ \widetilde{f}$ has
  the desired properties, for $\Psi_1 (\widetilde{f} (A_{[0, 1 - r_0]} \times
  P_0)) = A_{[0, 1 - r_0]} \times P_0$ and $\widetilde{f}_{|A_{[1 - r_0, 1]} \times
  S^1} = \tmop{id}$.
\end{proof}

\subsection{The Fifth Step}

\begin{theorem}
  Let $f$ be a self-diffeomorphism of $\mathcal{M}= D^2 \times S^1$ such that
  $f (D^2 \times P_0) = D^2 \times P_0$ and $f = \tmop{id}$ on a neighborhood
  of $\partial \mathcal{M}$. Then $f$ is ambient isotopic rel $\partial
  \mathcal{M}$ to a diffeomorphism $f'$ of $\mathcal{M}$ that is identical on
  a neighborhood of $\partial \mathcal{M} \cup (D^2 \times P_0)$.
\end{theorem}

\begin{proof}
  First, we claim that without loss of generality we may assume $f_{|D^2
  \times P_0} = \tmop{id}$. To see this, note that $f_{|D^2 \times P_0}^{- 1}$
  is a self-diffeomorphism of $D^2$ that is identical near $\partial D^2$. By
  \cite{Smale2sphere} Theorem 4 (p.624), there is an
  ambient isotopy $\tau : D^2 \times I \rightarrow D^2$ such that $\tau_1 = f_{|D^2
  \times P_0}^{- 1}$ and $\tau$ is stable in a neighborhood of $\partial D^2$. Let
  $T : \mathcal{M} \times I \rightarrow \mathcal{M}$ be the ambient isotopy
  of $\mathcal{M}= D^2 \times S^1$ defined by $T (x, z, t) = (\tau (x, t),
  z)$. Then $T$ is stable near $\partial \mathcal{M}$. Hence $f \circ
  T_1$ is identical near $\partial \mathcal{M}$ and on $D^2 \times P_0$,
  and we can simply replace $f$ by $f \circ T_1$.
  
  By assumption, there exists a $r_0 > 0$ such that $f_{|A_{[1 - r_0, 1]}
  \times S^1} = \tmop{id}$, where $A_J = \{ x \in D^2 | \norm{x} \in J \}$.
  
  Regard $D^2 \times \mathbb{R}$ as the trivial vector bundle over $D^2$.
  Define $C \of D^2 \times \mathbb{R} \rightarrow D^2 \times S^1$ by $C (x, t)
  = (x, e^{(\arctan t) i})$. Then $C$ and $f \circ C$ are both tubular
  neighborhoods of $\tmop{Int} D^2 \times P_0$ in $\tmop{Int} \mathcal{M}$. We
  shall construct an isotopy from $f \circ C$ to $C$ that is stable on $A_{[1
  - 2 r_0 / 3, 1)} \times \mathbb{R}$ by modifying the proof of \cite{Hirsch}
  Theorem 4.5.3.
  
  By compactness of $D^2$, there exists $\varepsilon > 0$ such that $f \circ C
  (D^2 \times [- \varepsilon, \varepsilon]) \subset \tmop{Im} C$. Choose a
  smooth function $\rho : I \rightarrow I$ that equals 0 on $[0, 1 - r_0]$ and
  equals 1 on $[1 - 2 r_0 / 3, 1]$. It follows from the monotonicity of the
  real valued function $s \mapsto \frac{s}{\sqrt{1 + s^2}}$ that for any $x
  \in \tmop{Int} D^2, t \in I$, the function $\sigma_{x, t} : \mathbb{R}
  \rightarrow \mathbb{R}$ defined by
  \[ \sigma_{x, t} (s) = \left( (1 - t) + t \left( \rho (\norm{x}) + (1 - \rho
     (\norm{x})) \frac{\varepsilon}{\sqrt{1 + s^2}} \right) \right) s \]
  is strictly increasing. Note that $\rho (\norm{x})$ is smooth on $x$ since
  $\rho (\norm{x}) \equiv 0$ near the origin of $\mathbb{R}^2$. Therefore the map
  $\Psi : \tmop{Int} D^2 \times \mathbb{R} \times I \rightarrow \tmop{Int} D^2
  \times \mathbb{R}, \Psi (x, s, t) \assign (x, \sigma_{x, t} (s))$ is a smooth
  isotopy, and so is $f \circ C \circ \Psi$.
  
  We make the following observations about the isotopy $f \circ C \circ \Psi$
  from $f \circ C_{| \tmop{Int} D^2 \times \mathbb{R}}$ to $f \circ C \circ \Psi_1$:
  
  (i) $\Psi$ is stable on $A_{[1 - 2 r_0 / 3, 1]} \times
  \mathbb{R}$
  
  (ii) If $\norm{x} \geqslant 1 - r_0$, then $f \circ C \circ \Psi_1 (x, s) = f
  \circ C (x, \sigma_{x, 1} (s)) = C (x, \sigma_{x, 1} (s)) \in \tmop{Im} C$
  
  (iii) If $\norm{x} \leqslant 1 - r_0$, then $f \circ C \circ \Psi_1 (x, s) = f
  \circ C \left( x, \frac{\varepsilon s}{\sqrt{1 + s^2}} \right) \in \tmop{Im}
  C$, since $\left| \frac{s}{\sqrt{1 + s^2}} \right| < 1$
  
  (iv) For any $t \in I$, $\Psi_t = \tmop{id}$ on $\tmop{Int} D^2 \times {0}$
  
  Define $g = C^{- 1} \circ f \circ C \circ \Psi_1 : \tmop{Int} D^2 \times
  \mathbb{R} \rightarrow \tmop{Int} D^2 \times \mathbb{R}$, and define $H :
  \tmop{Int} D^2 \times \mathbb{R} \times I \rightarrow \tmop{Int} D^2 \times
  \mathbb{R}$ by
  \begin{eqnarray*}
    H (x, s, t) = & \left\{\begin{array}{l}
      t^{- 1} g (t (x, s)) = t^{- 1} g (x, t s), t \in (0, 1]\\
      \Phi (x, s), t = 0
    \end{array}\right. & 
  \end{eqnarray*}
  where $\Phi$ is the fiber derivative of $g$ (this makes sense since $g= \tmop{id}$ on $\tmop{Int} D^2 \times {0}$). The proof of \cite{Hirsch} 4.5.3
  shows that $H$ is smooth. Moreover, when $\norm{x} \geqslant 1 - \frac{2
  r_0}{3}, t > 0$, we can compute:
  \begin{align*}
    & H (x, s, t) & &\\
    = & t^{- 1} \cdot g (x, t s) & &\\
    = & t^{- 1} \cdot C^{- 1} \circ f \circ C \circ \Psi_1 (x, t s) & &\\
    = & t^{- 1} \cdot C^{- 1} \circ f \circ C (x, t s) & &\tmop{by}
    \tmop{stability} \tmop{of} f \circ C \circ \Psi \tmop{on} A_{[1 - 2 r_0 / 3,
    1]} \times \mathbb{R}\\
    = & t^{- 1} \cdot C^{- 1} \circ C (x, t s) & &\tmop{for} f_{|A_{[1 - r_0,
    1]} \times S^1} = \tmop{id}\\
    = & (x, s) & &
  \end{align*}
  By continuity, $H (x, s, t) = (x, s)$ whenever $\norm{x} \geqslant 1 - \frac{2
  r_0}{3}$. Therefore the isotopy $C \circ H$ between $C \circ \Phi$ and $f
  \circ C \circ \Psi_1$ is stable on $A_{[1 - 2 r_0 / 3, 1)} \times \mathbb{R}$.
  
  Since $g$ is a local diffeomorphism, the fiber derivative $\Phi$ is a vector
  bundle isomorphism. Hence it is of the form $\Phi (x, s) = (x, \varphi (x)
  s)$ where $\varphi : \tmop{Int} D^2 \rightarrow \mathbb{R}- \{ 0 \}$ is a
  smooth function that equals 1 on $A_{[1 - 2 r_0 / 3, 1)}$. Connectivity
  implies $\varphi (x) > 0$ for all $x \in \tmop{Int} D^2$. Thus we can bridge
  $C \circ \Phi$ and $C_{| \tmop{Int} D^2 \times \mathbb{R}}$ by an isotopy of the form $C (x, s, t) \mapsto (x, ((1
  - t) + t \varphi (x)) s)$, which is obviously stable on $A_{[1 - 2 r_0 / 3,
  1)} \times \mathbb{R}$. Combining this isotopy with $C \circ H$ and $f \circ C
  \circ \Psi$ in the appropriate order and direction, we obtain (after smoothing
  the joints) a smooth isotopy
  \begin{equation*}
  F: \tmop{Int}D^2 \times \mathbb{R} \times I \rightarrow \tmop{Int}D^2 \times S^1
  \end{equation*}
from $f \circ C_{| \tmop{Int} D^2 \times \mathbb{R}}$ to $C_{| \tmop{Int} D^2 \times \mathbb{R}}$ that is stable on
  $A_{[1 - 2 r_0 / 3, 1)} \times \mathbb{R}$.
  
  {\tmstrong{Claim I}}: The set $F (A_{[0, 1 - 2 r_0 / 3)} \times \mathbb{R}
  \times I)$ lies in $A_{[0, 1 - 2 r_0 / 3)} \times S^1$.
  
  Suppose on the contrary there is some $x_0 \in A_{[0, 1 - 2 r_0 / 3)}, s_0
  \in \mathbb{R}, t_{0} \in I$ such that $F (x_0, s_0, t_0) \in A_{[1 - 2 r_0 /
  3, 1)} \times S^1$. The image of the isotopy $f \circ C \circ \Psi$ lies in
  $\tmop{Im} f \circ C$, while the image of the other two isotopies composing
  $F$ lies in $\tmop{Im} C$. We consider the two cases separately:
  
  Case 1: If $F (x_0, s_0, t_0) = C (x_1, s_1)$ for some $(x_1, s_1) \in \tmop{Int} D^2
  \times \mathbb{R}$. Then $C (x_1, s_1) \in A_{[1 - 2 r_0 / 3, 1)} \times
  S^1$ implies $x_1 \in A_{[1 - 2 r_0 / 3, 1)}$. By the stability of $F$ on
  $A_{[1 - 2 r_0 / 3, 1)} \times \mathbb{R}$, we have $F (x_0, s_0, t_0) = C
  (x_1, s_1) = F (x_1, s_1, 1) = F (x_1, s_1, t_0)$. But $F$ is an isotopy,
  thus we must have $x_0 = x_1$, a contradiction.
  
  Case 2: If $F (x_0, s_0, t_0) = f \circ C (x_2, s_2)$ for some $(x_2, s_2)
  \in \tmop{Int} D^2 \times \mathbb{R}$. Since $f$ is identical on $A_{[1 - r_0, 1]}
  \times S^1$, we know $F (x_0, s_0, t_0) = f^{- 1} \circ F (x_0, s_0, t_0) =
  f^{- 1} \circ f \circ C (x_2, s_2) = C (x_2, s_2)$, and the argument in Case
  1 leads to a contradiction.
  
  This proves the claim above.
  
  Our next step is to construct an ambient isotopy $\bar{F}$ of $\mathcal{M}$
  that equals $F$ in a neighborhood of $\tmop{Int}D^2 \times P_0$. Let $\alpha :
  \mathcal{M} \rightarrow I$ be a smooth function that equals $0$ on $A_{[1 -
  r_0 / 3, 1]} \times S^1$ and 1 on $A_{[0, 1 - 2 r_0 / 3]} \times S^1$. Take
  a $C^{\infty}$ 3-ball $B$ in $\tmop{Int} D^2 \times \mathbb{R}$ that
  contains $A_{[0, 1 - r_0 / 3]} \times \{ 0 \}$ in its interior. Apply
  Lemma \ref{VF}  to the isotopy $\widetilde{F} \assign F \circ ((f^{- 1} \circ C^{-
  1}) \times \tmop{id}_I)$ of $f \circ C (B)$ in $\tmop{Int} \mathcal{M}$, we
  obtain a compactly supported vector field $X$ on $\tmop{Int} \mathcal{M}
  \times I$ that equals to the tangent vectors of curves $t \mapsto \widehat{F}
  (x, t)$ on $\tmop{Im} \widehat{F}$, where $\widehat{F}$ denotes the track of
  $\widetilde{F}$. Define a time-dependent vector field $G$ by composing $X$ with
  the map $T p : T (\tmop{Int} \mathcal{M} \times I) \rightarrow T \tmop{Int}
  \mathcal{M}$ induced by the projection $p : \tmop{Int} \mathcal{M} \times I
  \rightarrow \tmop{Int} \mathcal{M}$. Define another time-dependent vector
  field $G'$ on $\tmop{Int} \mathcal{M}$ by $G' (x, t) = \alpha (x) G (x, t),
  x \in \tmop{Int} \mathcal{M}, t \in I$. Then $G' = 0$ on $A_{[1 - r_0 / 3,
  1)} \times S^1 \times I$.
  
  {\tmstrong{Claim II}}: The time-dependent vector fields $G$ and $G'$
  coincide on $\tmop{Im} \widehat{F}$.
  
  An element of $\tmop{Im} \widehat{F}$ is of the form $\widehat{F} (f \circ C (y, s),
  t)$ where $(y, s) \in B$. By definition, we have $\widehat{F} (f \circ C (y, s),
  t) = (\widetilde{F} (f \circ C (y, s), t), t) = (F (y, s, t), t)$. There are two
  cases:
  
  Case 1: If $y \in A_{[0, 1 - 2 r_0 / 3)}$, then the Claim I we proved
  earlier implies $F (y, s, t) \in A_{[0, 1 - 2 r_0 / 3)} \times S^1$.
  Therefore $\alpha \circ F (y, s, t) = 1$ and $G' (F (y, s, t), t) = \alpha
  \circ F (y, s, t) \cdot G (F (y, s, t), t) = G (F (y, s, t), t)$.
  
  Case 2: If $y \in A_{[1 - 2 r_0 / 3, 1)}$, then we compute:
  \begin{align*}
    & X \circ \widehat{F} (f \circ C (y, s), t) & &\\
    = & \frac{\partial}{\partial t} \widehat{F} (f \circ C (y, s), t) & &\\
    = & \frac{\partial}{\partial t} (F (y, s, t), t) & &\\
    = & \left( 0, \frac{\partial}{\partial t} \right) & &\tmop{by}
    \tmop{stability} \tmop{of} F \tmop{on} A_{[1 - 2 r_0 / 3, 1)} \times
    \mathbb{R}\\
  \end{align*}
  Thus $G \circ \widehat{F} (f \circ C (y, s), t) = 0 = G' \circ \widehat{F} (f \circ
  C (y, s), t)$ in Case 2. This proves Claim II.
  
  As before, the time-dependent vector field $G'$ generates an ambient isotopy
  of $\tmop{Int} \mathcal{M}$. This isotopy is supported in $A_{[0, 1 - r_0 /
  3]} \times S^1$, and thus we can extend it to an ambient isotopy of
  $\mathcal{M}$ by identity. Denote the extended isotopy by $\bar{F}$. From
  Claim II, we know $\bar{F}$ extends $\widetilde{F}$.
  
  Choose $b > 0$ such that $A_{[0, 1 - r_0 / 3]} \times (- b,
  b) \subset \tmop{Int} B$.
  
  {\tmstrong{Claim III}}: that $\bar{F}_1 \circ f$ coincide with identity on
  $C (D^2 \times (- b, b))$.
  
  There are two cases:
  
  If $(y, s) \in A_{[1 - r_0 / 3, 1]} \times (- b, b)$, then
  \begin{align*}
    & \bar{F}_1 \circ f \circ C (y, s) & &\\
    = & \bar{F}_1 \circ C (y, s) & & \tmop{for} f_{|A_{[1 - r_0, 1]} \times S^1}
    = \tmop{id}\\
    = & C (y, s) & & \tmop{for} \bar{F} \tmop{is} \tmop{stable} \tmop{on} A_{[1 - r_0 / 3, 1]} \times S^1
  \end{align*}
  
  If $(y, s) \in A_{[0,1 - r_0 / 3]} \times (- b, b)$, then
  \begin{align*}
    & \bar{F}_1 \circ f \circ C (y, s) & & \\
    = & \widetilde{F}_1 \circ f \circ C (y, s) & & \tmop{for} \bar{F} \tmop{extends}
    \widetilde{F}\\
    = & F_1 (y, s) & &\tmop{by} \tmop{definition} \tmop{of} \widetilde{F}\\
    = & C (y, s) & &
  \end{align*}
  This proves Claim III. Now $f' = \bar{F}_1 \circ f$ is the desired
  diffeomorphism.
\end{proof}

\subsection{The Sixth Step}

\begin{theorem}
  Let $f$ be a diffeomorphism of $\mathcal{M}= D^2 \times S^1$ that is
  identical on a neighborhood of $\partial \mathcal{M} \cup (D^2 \times P_0)$.
  Then $f$ is ambient isotopic rel $\partial \mathcal{M}$ to identity.
\end{theorem}

\begin{proof}
  Let $W$ be a neighborhood of $\partial \mathcal{M} \cup (D^2 \times P_0)$
  such that $f_{|W} = \tmop{id}$. Take a smooth 3-ball $B \subset \tmop{Int}
  \mathcal{M}- D^2 \times P_0$ such that the closure $\overline{\mathcal{M}-
  W} \subset \tmop{Int} B$. Now $f_{|B}$ is a diffeomorphism of $B$ that
  equals to identity near $\partial B$. By Theorem \ref{SC}, Theorem \ref{PathI} and Theorem
  \ref{PathII}, there exist an ambient isotopy $F : B \times I \rightarrow I$ such that
  $F_1 = f$ and $F$ is stable in a neighborhood of $B$. Now extend $F$ by
  identity to an ambient isotopy.
\end{proof}

\section{Proof of Theorem \ref{Handlebody} and Theorem \ref{MainTheoremSmooth}}
\begin{proof} [Proof of Theorem \ref{Handlebody}]

We will prove by induction. The case $g=1$ is given by Theorem \ref{Main}. The proof of the inductive step is basically parallel to that of Theorem \ref{Main}, and we shall merely present the necessary modifications. Assume that the conclusion holds for $g \leq n-1$. Given a handlebody $V_g$ and a diffeomorphism $f$ rel $\partial V_g$. We could view the handles as images of smooth embeddings of $D^2 \times [-2,2]$ into $V_g$. Let $i_0$ be one of these embeddings. Note that $i_0(D^2 \times [-2,2]) \cap \partial V_g =i_0(\partial D^2 \times [-2,2])$.

Define a ``partial collar'' $\Theta_0 : \partial D^2 \times [- 2, 2] \times
\left[ 0, \frac{1}{2} \right] \hookrightarrow V_g$ by
\[ \Theta_0 (z, t, r) = i_0 ((1 - r) z, t), z \in \partial D^2, t \in [- 2, 2],
   r \in \left[ 0, \frac{1}{2} \right] \]
where $D^2$ is regarded as a subset of $\mathbb{C}$.

Denote by $i$ the restriction of $i_0$ to $D^2 \times [- 1, 1]$. Mimicking the
proof of Lemma \ref{partialextension}, we can construct a ``partially extension'' of $\Theta_0$,
that is, a collar $\Theta : \partial V_g \times [0, d) \hookrightarrow V_g$ for
some $d > 0$ such that $\Theta (i (z, t), r) = \Theta_0 (z, t, r) = i ((1 - r) z,
t)$ for all $z \in \partial D^2, t \in [- 1, 1], r \in [0, d)$. Shrinking $d$
if necessary, we may further assume $\tmop{Im} \Theta \cap \tmop{Im} i = i
(A_{(1 - d, 1]} \times [- 1, 1])$.

Objects that appear in the proof of \ref{Main} should be replaced by their apparent
analogy, we list a few as examples:

1. $\mathcal{M}$ or $D^2 \times S^1$ by $V_g$.

2. $T_J$ by $i (D^2 \times J)$ for any interval $J \subset [- 1, 1]$. In
particular, $D^2 \times P_0$ is replaced by $i (D^2 \times \{ 0 \})$

3. $A_J \times P_0$ by $i (A_J \times \{ 0 \})$ for any interval $J$

4. $A_{[1 - a, 1]} \times S^1$ by $\Theta (\partial V_g \times [0, a])$

5. $A_{[0, a]} \times S^1$ by $V_g - \Theta (\partial V_g \times (1 - a, 1])$

\

We will define another collar $\Theta'$ of $\partial V_g$ that is identical to
$\Theta$ except twisted near $D^2 \times \{ 0 \}$. Choose a smooth function
$\lambda : [- 1, 1] \rightarrow [0, 1]$ with

(i) $\lambda \equiv 0$ near the endpoints $- 1$ and $1$

(ii) $\lambda \equiv \varepsilon_0$ near 0 for some constant $\varepsilon_0 >
0$

(iii) $| \lambda' (t) | < \frac{1}{d}$ for all $t \in [- 1, 1]$

We now give an analogy of the First Step (Theorem \ref{1st}).

Define $\Theta' : \partial V_g \times [0, d) \hookrightarrow V_g$ by
\[ \Theta' (x, r) = \left\{\begin{array}{l}
     i ((1 - r) z, t + \lambda (t) r), \tmop{if} x = i (z, t), z \in \partial
     D^2, t \in [- 1, 1]\\
     \Phi (x, r), \tmop{otherwise}
   \end{array}\right. \]
That $\Theta'$ is a smooth collar is guaranteed by property (i) and (iii) of
$\lambda$. By Theorem \ref{IUC}, there exists some $r_0 \in (0, d)$ and an ambient
isotopy $G$ of $V_g$ rel $\partial V_g$ such that
\[ G_1 \circ f \circ \Theta (x, r) = \Theta' (x, r), x \in \partial V_g, 0
   \leqslant r \leqslant r_0 \]
The proof of the Second Step (Theorem \ref{2nd}) is valid with the above-mentioned
replacements. Hence $G_1 \circ f$ is ambient isotopic rel $\partial V_g$ to
some $f'$ such that
\begin{equation}
  f' \circ \Theta (x, r) = \Theta' (x, r), x \in \partial V_g, 0 \leqslant r
  \leqslant r_0
\end{equation}
for possibly smaller $r_0$ and that $f'_{|D^2 \times \{ 0 \}}$ is transverse
to $D^2 \times \{ 0 \}$.

The proof of the Third Step (Theorem \ref{3rd} and Lemma \ref{Lemmaforproof}) works equally well in
the induction step. As a result, $f$ is ambient isotopic rel $\partial V_g$
to some $f'$ that satisfy $(2)$ for even smaller $r_0$ and such that $f' (D^2
\times \{ 0 \}) \cap (D^2 \times \{ 0 \}) = \partial D^2 \times \{ 0 \}$.

The proof of the Fourth Step (Theorem \ref{4th}) should be modified as follows: When
choosing $r_0$, we require that $3 r_0 < d$, that $f$ satisfy $(2)$ for $0
\leqslant r \leqslant 3 r_0$ and
\[ \Theta (\partial V_g \times [0, 3 r_0]) \cap f (i (D^2 \times \{ 0 \})) = f
   (i (A_{[1 - 3 r_0, 1]} \times \{ 0 \})) \]
The map $\tilde{H}$ has to be redefined. Let $\tilde{H} : i (D^2 \times [- 1,
1]) \times I \rightarrow i (D^2 \times [- 1, 1])$ be defined by
\[ \tilde{H} (i (w, t), s) = i (w, t + s \lambda (t) \rho (| w |)) \]
Our choice of $\lambda$, $r_0$ and $\rho$ ensures that $\tilde{H}$ is an
isotopy and can be extended to an ambient isotopy of $V_g$ by identity. The
remainder of the proof carries verbatim, and we obtain a $f'$ ambient isotopic
to $f$ rel $\partial V_g$ such that $f' (i (D^2 \times \{ 0 \})) = i (D^2
\times \{ 0 \})$ and $f' = \tmop{id}$ near $\partial V_g$.

For the Fifth Step, we need to redefine the ambient isotopy $T : V_g \times I
\rightarrow V_g$ by defining $T : i (D^2 \times [- 1, 1]) \times I \rightarrow
i (D^2 \times [- 1, 1])$ by
\[ T (i (w, t), s) = i \left( \tau \left( w, \frac{s \lambda
   (t)}{\varepsilon_0} \right), t \right) \]
and extending by identity. The map $C : D^2 \times \mathbb{R} \rightarrow V_g$
should be redefined by $C (x, t) = i \left( x, \frac{2}{\pi} \tmop{arctant} t
\right)$. Note that the condition $\tmop{Im} \Phi \cap \tmop{Im} i = i (A_{(1
- d, 1]} \times [- 1, 1])$ is needed to prove Claim I. As a consequence of
this step, we obtain a $f'$ ambient isotopic to $f$ rel $\partial V_g$ that is
identical near $\partial V_g \cup i (D^2 \times \{ 0 \})$. Now we can choose a
handlebody $V_{g - 1}$ embedded in $\tmop{Int} V_g$ such that $f'$ is
identical on $V_g - V_{g - 1}$ and in a neighborhood of $\partial V_{g - 1}$
in $V_{g - 1}$. The inductive assumption implies that $f'_{|V_{g - 1}}$ is
ambient isotopic to identity rel $\partial V_{g - 1}$. By Theorem \ref{PathI} and
Theorem \ref{PathII}, there is an ambient isotopy from $f'_{|V_{g - 1}}$ to identity
through diffeomorphisms in $\tmop{Diff}_{\partial, U} (V_{g - 1})$. Thus $f'$
is ambient isotopic to identity rel $\partial V_g$, and the same is true for
$f$.
\end{proof}
\begin{proof} [Proof of Theorem \ref{MainTheoremSmooth}]
Let $f$ be a self diffeomorphism of $V_g$ such that $f_{| \partial V_g}$ is isotopic to identity. Then the same is true for $f^{-1}$, and there exist an ambient isotopy $F : \partial V_g \times I \rightarrow \partial V_g$ such that $F_1 = f^{-1}$. Choose a smooth $\zeta :I \rightarrow I$ such that

(i) $\zeta \equiv 1$ near 0

(ii) $\zeta \equiv 0$ near 1

Let $C : \partial V_g \times I \rightarrow V_g$ be a smooth collar of $\partial V_g$. The map $G : C(\partial V_g \times I) \times I \rightarrow C(\partial V_g \times I)$ defined by 
\[ G(C(x,s),t)= C(F(x,t \zeta(s))) \]
is an ambient isotopy of $C(\partial V_g \times I)$ that can be extended by identity to one of $V_g$, which we still denote as $G$. Since $G_{1 |\partial V_g}=f^{-1}$, $f \circ G$ isotope $f$ to a diffeomorphism of $V_g$ that is identical on $\partial V_g$. Now apply Theorem \ref{Handlebody}.
\end{proof}
\newpage
\bibliographystyle{plain}
\bibliography{DwindledBible}

Fang Sun,

School of Mathematical Sciences,

Capital Normal University,

105 West Third Ring Road North, Haidian District,

Beijing,

100048,

China

email: fsun@cnu.edu.cn\\

Xuezhi Zhao

School of Mathematical Sciences,

Capital Normal University,

105 West Third Ring Road North, Haidian District,

Beijing,

100048,

China

email: zhaoxve@mail.cnu.edu.cn

\end{document}